\documentclass[a4paper,10pt,twoside]{article}
\pdfoutput=1

\usepackage[activate={true,nocompatibility},final,tracking=true,kerning=true,spacing=true,factor=1100,stretch=0,shrink=0]{microtype}
 \SetTracking{encoding={*}, shape=sc}{30}
 \microtypecontext{spacing=nonfrench}

\usepackage[T1]{fontenc} 
\usepackage[utf8]{inputenc}
\usepackage[british]{babel}
\usepackage{mathrsfs}
\usepackage{amsfonts}
\usepackage{amsmath} 
\usepackage{amsthm}  
\usepackage{amssymb}
\usepackage{mathtools}
\usepackage{graphicx}
\usepackage{xcolor}
\usepackage{url}
\usepackage[noadjust]{cite}
\usepackage{stmaryrd}
\usepackage{titling}
\usepackage{centernot}
\usepackage{indentfirst}

\usepackage[hang,flushmargin]{footmisc}
\usepackage[affil-it]{authblk}
\usepackage{hyperref}
\hypersetup{colorlinks=false, pdfborder={0 0 0}}

\usepackage[inline]{enumitem}
\usepackage{fancyhdr}
\usepackage[capitalize]{cleveref}
\Crefname{pr}{Proposition}{Propositions}
\Crefname{subsection}{Subsection}{Subsections}

\usepackage{tikz}
\usetikzlibrary{matrix,arrows,decorations.pathmorphing, shapes, decorations.fractals, patterns, positioning}
\tikzset{snake it/.style={decorate, decoration={snake, segment length=5.3pt, amplitude=0.25mm}}}

\newcommand{\from}{\colon}
\newcommand{\implica}{\rightarrow}
\renewcommand{\phi}{\varphi}
\renewcommand{\epsilon}{\varepsilon}
\renewcommand{\models}{\vDash}
\newcommand{\proves}{\vdash}
\newcommand{\restr}{\upharpoonright}
\newcommand{\monster}{\mathfrak U}
\newcommand{\invtypes}{S^{\mathrm{inv}}}
\newcommand{\invext}{\mid}
\newcommand{\bla}[4]{{#1}_{#2}#3\ldots#3{#1}_{#4}}
\newcommand{\pow}[2]{#1^{(#2)}}
\newcommand{\inverse}{^{-1}}
\newcommand{\allora}{\Rightarrow}
\newcommand{\then}{\Longrightarrow}
\DeclareMathOperator{\alt}{alt}
\DeclareMathOperator{\tp}{tp}
\DeclareMathOperator{\cl}{cl}
\DeclareMathOperator{\aut}{Aut}

\DeclareMathOperator{\orb}{orb}
\DeclareMathOperator{\im}{im}
\DeclareMathOperator{\sign}{sign}
\DeclareMathOperator{\supp}{supp}
\DeclarePairedDelimiter{\set}{\{}{\}}
\DeclarePairedDelimiter{\abs}{\lvert}{\rvert}
\DeclarePairedDelimiter{\seq}{\langle}{\rangle}
\theoremstyle{definition}
\newtheorem{defin}{Definition}[section]
\newtheorem{thm}[defin]{Theorem}
\newtheorem{pr}[defin]{Proposition}
\newtheorem{co}[defin]{Corollary}
\newtheorem{lemma}[defin]{Lemma}
\newtheorem{notation}[defin]{Notation}
\newtheorem{eg}[defin]{Example}
\newtheorem{rem}[defin]{Remark}
\newtheorem{fact}[defin]{Fact}
\newtheorem{ass}[defin]{Assumption}
\newtheorem{question}[defin]{Question}
\newtheorem{prob}[defin]{Problem}
\newtheorem{claim}{Claim}[defin]
\newtheorem{alphthm}{Theorem}

\let\oldqed\qedsymbol
\newcommand{\qedclaim}{\mbox{$\underset{\textsc{claim}}{\oldqed}$}}
\makeatletter
\newenvironment{claimproof}[1][\it Proof of Claim]{
\let\qedsymbol\qedclaim
  \par
  \pushQED{\qed}%
  \normalfont \topsep6\p@\@plus6\p@\relax
  \trivlist
\item[\hskip\labelsep
  \upshape
  #1\@addpunct{.}]\ignorespaces
}{%
  \popQED\endtrivlist\@endpefalse
}
\makeatother
\let\qedsymbol\oldqed

\makeatletter
\newcommand{\subjclass}[2][2020]{%
  \let\@oldtitle\@title%
  \gdef\@title{\@oldtitle\footnotetext{\hspace*{-2em}#1 \emph{Mathematics subject classification.} #2}}%
}
\newcommand{\keywords}[1]{%
  \let\@@oldtitle\@title%
  \gdef\@title{\@@oldtitle\footnotetext{\hspace*{-2em}\emph{Keywords:} #1.}}%
}
\makeatother

\author[1]{Jan Dobrowolski%
\thanks{email: \url{dobrowol@math.uni.wroc.pl} \textsc{orcid}: \url{https://orcid.org/0000-0003-3435-4782}}}
\affil[1]{\small Department of Mathematics, University of Manchester, Oxford Road, Manchester M13 9PL, United Kingdom}
\affil[1]{\small Instytut Matematyczny\\
Uniwersystet Wroc\l{}awski\\
Wroc\l{}aw\\
Poland}

\author[2]{Rosario Mennuni%
\thanks{email: \url{R.Mennuni@posteo.net} \textsc{orcid}: \url{https://orcid.org/0000-0003-2282-680X}}}
\affil[2]{\small Dipartimento di Matematica, Universit\`a di Pisa, Largo Bruno Pontecorvo 5, 56127 Pisa, Italy}
\title{The Amalgamation Property for automorphisms of ordered abelian groups}

\date{}



\begin{document}

\maketitle

\vspace{-1.5cm}

\begin{abstract}
  We prove that the category of ordered abelian groups equipped with an automorphism has the Amalgamation Property, deduce that their inductive theory is $\mathsf{NIP}$ in the sense of positive logic, and initiate a development of the latter framework. 

As byproducts of the proof, we obtain a generalised version of the Hahn Embedding Theorem which allows to lift each automorphism of an ordered abelian group to one of an ordered real vector space, and we show that, on existentially closed structures, linear combinations of iterates of the automorphism have the Intermediate Value Property.
\end{abstract}

\tableofcontents

\bigskip

One of the first mathematical skills that are learnt by students in elementary schools is how to add and subtract integers, and how to compare their sizes. Whilst most teachers may not explicitly pronounce the words ``ordered abelian group'', they are indeed teaching their pupils to work in them. More relevantly for current mathematical research, ordered abelian groups (henceforth, \emph{oags}) are fundamental objects in the algebra of ordered structures, in tropical geometry, and  are crucial in the study of (not necessarily ordered) valued fields, where they occur in the form of value groups. 
Their non-abelian analogues appear in geometry as braid groups~\cite{Dehornoy_1994,Fenn_Greene_Rolfsen_Rourke_Wiest_1999}, in group theory as right-angled Artin groups~\cite{Duchamp_Thibon_1992} and, sometimes, even as automorphism groups of abelian ones \cite{Conrad_1958}.

Perhaps usually absent in elementary education, that of automorphism is another fundamental concept, and another main ingredient of the present work. Automorphisms of oags have been  a topic of interest from the times of H\"older, who proved that  every \emph{Archimedean} oag, that is, one where every non-trivial cyclic subgroup is unbounded, may be embedded in the ordered additive group of the reals, and that every automorphism of such an oag may be identified with multiplication by a real number, to the more recent work of Laskowski and Pal~\cite{Laskowski_Pal_2015} and of Kuhlmann and Serra~\cite{ks_groups}. Applications to the study of automorphisms of valued fields appeared for instance in \cite{Chernikov_Hils_2014,Kuhlmann_Serra_2022}. Some decidability problems related to automorphisms of ordered abelian groups were studied in \cite{Mart}.

If automorphisms of Archimedean oags are easily classified, the situation changes drastically when we look at oags with more than one Archimedean class, that is, where there are positive elements $x$, $y$ such that for every natural number $n$ the $n$-fold sum of $x$ with itself lies below $y$. One important question is whether automorphisms of oags have the \emph{Amalgamation Property}, abbreviated as AP. The latter is a fundamental notion in several areas, ranging from category theory~\cite{kiss1983categorical}, to model theory~\cite{hodges}, where it is a crucial ingredient in the definition of a Fra\"iss\'e class, hence to permutation groups~\cite{Macpherson_2011} and, via the Kechris--Pestov--Todor\v cevi\'c correspondence~\cite{Kechris_Pestov_Todorcevic_2005}, to Ramsey theory and topological dynamics. Explicitly,  saying that automorphisms of oags have the AP means that, given any pair of embeddings of oags $f\from (A, \sigma_A)\to (B, \sigma_B)$ and $g\from (A, \sigma_A)\to (C, \sigma_C)$, commuting with the automorphisms $\sigma_-$, there are $(D, \sigma_D)$ and embeddings of $(B, \sigma_B)$ and $(C, \sigma_C)$ into it, again commuting with the involved $\sigma_-$, making the diagram in \Cref{fig:apdiag} commute.
\begin{figure}[b]
  \begin{center}
    \begin{tikzpicture}[scale=3]
  \node(a) at (0,0){$(A,\sigma_A)$};
  \node(b) at (1,0.5){$(B, \sigma_B)$};
  \node(c) at (1,-0.5){$(C, \sigma_C)$};
  \node(d) at (2,0){$(D, \sigma_D)$};
  \node(comm) at (1,0){$\circlearrowright$};
\path[->, thick,  font=\scriptsize,>= angle 90]
(a) edge node [above]  {$f$} (b)
(a) edge node [below]  {$g$} (c)
(b) edge [dashed] node [above]  {$\exists$} (d)
(c) edge [dashed] node [below]  {$\exists$} (d)
;
\end{tikzpicture}
\end{center}
\caption{An amalgamation diagram of automorphisms of oags.}\label{fig:apdiag}
\end{figure}
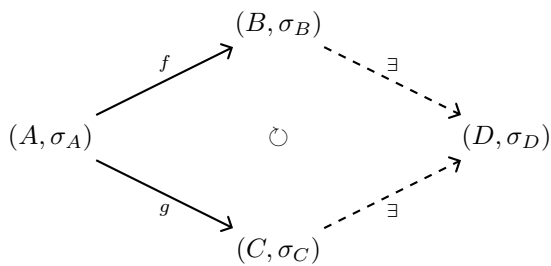

If we remove some structure, then the situation is well understood.
For instance, if we drop the group operation, that is, consider the case of automorphisms of linearly (that is, totally) ordered sets, the AP was proven in~\cite{Lascar}; in fact, the situation is also clear for several automorphisms, and for other kinds of operators~\cite{lipparini}. If instead we drop the automorphism, that is, consider pure oags, then not only the AP is known to hold, but there is a description of the posets of different amalgamations up to equivalence~\cite{Pierce_1972}. 
When we simultaneously have an abelian group structure, an order compatible with the latter, and an automorphism preserving both, to the best of our knowledge it was so far unknown whether the AP holds or not. Our main result is that this is indeed the case.

\begin{alphthm}[\Cref{thm:apqtor}]\label{athm:AP}
  The category of ordered abelian groups with an automorphism has the Amalgamation Property.
\end{alphthm}
Our proof uses model-theoretic methods, most prominently the Compactness Theorem, a pinch of o-minimality~\cite{van_den_dries_tame_1998}, and the notion of an \emph{existentially closed} structure (\Cref{defin:ec}).
In the case of an oag with an automorphism $(M, \sigma_M)$, this essentially means that whenever, in some extension $(B,\sigma_B)\supseteq (M, \sigma_M)$, a  subset  of $B^{nk}$ that is defined by affine equations and inequalities with parameters in $M$ contains a point of the form
\[
  (\bla a0,{n-1}, \sigma(a_0), \ldots , \sigma(a_{n-1}), \ldots, \sigma^{k-1}(a_0), \ldots ,\sigma^{k-1}(a_{n-1}))
\]
  then such a point, satisfying the same affine constraints, may already be found in $M^{nk}$. The model-theoretic preliminaries necessary for our proof of \Cref{athm:AP} are recalled, together with some standard notions around oags, in \Cref{sec:prelim}.

The method of proof we employed generated a couple of byproducts which we believe to be interesting in their own right, the first being the below generalisation of the Hahn Embedding Theorem, proven in \Cref{sec:realification}. Recall that the latter states that every oag $A$ may be embedded in an oag of generalised power series with real coefficients, let us call it $B$. Although every automorphism of $A$ may be extended to $B$, the  latter carries the structure of an ordered \emph{real} vector space, and one may need the extension of $\sigma$ to preserve the real structure. It turns out that, while the $B$ given by the Hahn Embedding Theorem is sometimes too small for this task to be achieved (\Cref{eg:nolift}), this is nevertheless possible by passing to a larger structure.
\begin{alphthm}[\Cref{thm:redtoR}]
  For every oag $A$  there are an ordered  $\mathbb R$-vector space $B$ and an embedding of oags  $A\to B$ such that every oag automorphism of $A$ lifts to an ordered $\mathbb R$-vector space automorphism of $B$.
\end{alphthm}

As for the second byproduct, recall that every continuous function on $\mathbb R$ has the \emph{Intermediate Value Property}: it maps convex sets to convex sets. In a general ordered abelian group with an automorphism $\sigma$, linear combinations of iterates of $\sigma$ are clearly continuous with respect to the order topology; nevertheless, this topology may fail to be connected, and it is indeed possible for convexity not to be preserved by such a function (\Cref{eg:noIVP}). Yet, we prove that on existentially closed structures, despite the failure of topological connectedness, the Intermediate Value Property survives. For technical reasons, we will actually need a slightly stronger statement.
\begin{alphthm}[\Cref{thm:minima}]\label{athm:ivp}
  If $(M, \sigma)$ is existentially closed, $\lambda_{i,j}$ are scalars, and $d_i\in M$, then the Intermediate Value Property holds for every function of the form $h(x)=\min_{i\le k} \left(d_i+\sum_{j\le n_i} \lambda_{i,j} \sigma^j(x)\right)$.
\end{alphthm}
After some preliminary work on \emph{absolutely monotone $\sigma$-polynomials} in \Cref{sec:OVSA}, we will prove \Cref{athm:ivp} in  \Cref{sec:ivp}, and use it to deduce \Cref{athm:AP} in \Cref{sec:amalgam}. This is in turn used in our final \Cref{sec:posnip} to prove that the \emph{positive theory} of automorphisms of oags is model-theoretically tame. More precisely, we prove it is \emph{dependent}, or \emph{$\mathsf{NIP}$}, and take the opportunity to develop the framework. 

In a nutshell, this means that existentially closed oags are combinatorially well-behaved according to a certain notion of tameness, provided the latter is formulated in a way that makes sense in this setting. In order to explain this in greater detail, and motivate our interest in it, we temporarily abandon automorphisms of oags in favour of automorphisms of fields. The class of existentially closed fields with an automorphism is \emph{elementary}, that is, it can be axiomatised by a first-order theory, known as $\mathsf{ACFA}$. Elementary classes may be studied under the lenses of \emph{neostability theory}, a branch of model theory concerned with the study of ``tameness properties'' that a first-order theory may or may not have. The theory $\mathsf{ACFA}$ satisfies one of these tameness conditions, known as \emph{simplicity}. This fact provides existentially closed difference fields with a canonical, well-behaved independence relation, and has had important applications in difference algebra~\cite{Chatzidakis_Hrushovski_1999,Hrushovski_2001,Medvedev_Scanlon_2014}.

One may ask whether this is also the case for existentially closed oags with an automorphism. An expert model theorist will notice at a glance that their simplicity is to be immediately ruled out, due to the presence of an infinite linear order, but that perhaps they may be \emph{dependent}. This means satisfying another tameness property, known as $\mathsf{NIP}$, strictly related to Vapnik--Chervonenkis dimension, hence to machine learning~\cite{Chase_Freitag_2019}, which had a crucial role, among other things, in the solution of Pillay's o-minimal conjectures~\cite{Hrushovski_Pillay_2011} and in the study of nowhere dense classes of graphs in computer science~\cite{nesetril_mendez_2016}. Classical examples of (non-simple) $\mathsf{NIP}$ theories include algebraically closed valued fields, the exponential ordered real field, fields of $p$-adic numbers, and the differential field of transseries~\cite{adh17}. The standard reference for the subject is~\cite{simon_2015}.

 Unfortunately, the expert model theorist will also recognise that the presence of an infinite linear order poses a more fundamental threat to the plan above: namely, by a theorem of Kikyo and Shelah~\cite{kikyo_strict_2002}, the class of existentially closed oags with an automorphism is not elementary.

 The latter obstruction may be bypassed by moving to \emph{positive logic}. Introduced in model theory in \cite{BY, BYP},  and known to the category theorists as \emph{coherent logic} (see~\cite{kamsma_type_2022} for the connection),  positive logic provides a framework that is broader than the  classical first-order one. In a certain precise sense, passing from classical to positive logic  corresponds to generalising from Boolean algebras to distributive lattices or, dually, from Stone spaces to spectral spaces~\cite{haykazyan_spaces_2019}. 

Generalisations of tameness properties to this setting began with Shelah pointing out~\cite{shelah_lazy_1975} that stability theory could be generalised from elementary classes to suitable classes of existentially closed structures. Pillay showed~\cite{pillay_forking_2000} that the same could be done for simplicity, and Ben Yaacov realised~\cite{ben-yaacov_simplicity_2003} that one may more generally work in positive logic. Recent work~\cite{DK, haykazyan_existentially_2021} extended to this larger framework the theory of Kim-independence in $\mathsf{NSOP}_1$ theories; a notion of an $\mathsf{NTP}_2$ theory was also considered in this context~\cite{haykazyan_existentially_2021, delbee_existentially_2021}. For more recent developments, see~\cite{dgk}.

 Our final \Cref{sec:posnip} initiates the development of $\mathsf{NIP}$ for positive theories, which was in fact the original motivation of our work. Besides the application of  \Cref{athm:AP},
this section may be read independently of the previous ones, and is structured as follows. After briefly recalling the setting in \Cref{sec:poslogprelim,sec:pilrob},
we show in \Cref{sec:nipbasics} that, by using a natural definition of $\mathsf{NIP}$, which is preserved by adding hyperimaginary sorts, basic results such as closedness under (positive) Boolean combinations, reduction to one variable, and equivalence with finiteness of alternation numbers generalise. In \Cref{sec:DLOA} we look at examples. First, for every group $G$, existentially closed $G$-actions by automorphisms on a dense linear order are positively dependent; this can be proven by using \cite[Th\'eor\`eme~4.4.19]{dealdama}, but we provide a quick direct argument. More interestingly, as anticipated, we also show positive dependence for existentially closed oags with an automorphism, by using \Cref{athm:AP}. 
\begin{alphthm}[\Cref{co:apthennip}]
  The class of existentially closed oags with an automorphism is $\mathsf{NIP}$ in the sense of positive logic.
\end{alphthm}
In \Cref{sec:invtp} we study invariant types, prove a version of Borel definability, and investigate two different notions of coheir extension. Under a mild assumption known as \emph{thickness}, we generalise the equivalence of $\mathsf{NIP}$ with the existence of at most $2^{\abs M}$  global $M$-invariant extensions of a fixed type over $M$.

\paragraph{Acknowledgements}
We thank P.~Freni and M.~Hils for the useful conversations. In particular, we thank the latter for a comment which led to \Cref{pr:famighnprp}. We are also grateful to Tomasz Rzepecki for an observation which made the argument in \Cref{claim:rho} more precise, and to Martin Ziegler for asking whether the ``moreover'' part of \Cref{thm:redtoR} held.

J.~Dobrowolski was supported by EPSRC Grant EP/V03619X/1 and by DFG project BA 6785/2-1.

R.~Mennuni was supported by the Italian research project  PRIN 2017: ``Mathematical Logic: models, sets, computability'' Prot.~2017NWTM8RPRIN and by the German Research Foundation (DFG) via HI 2004/1-1 (part of the French-German ANR-DFG project GeoMod) and under Germany’s Excellence Strategy EXC 2044-390685587, `Mathematics M\"unster: Dynamics-Geometry-Structure', is a member of the INdAM research group GNSAGA, and acknowledges the MIUR Excellence Department Project awarded to the Department of Mathematics, University of Pisa, CUP I57G22000700001.

\section{Preliminaries}\label{sec:prelim}
Throughout the paper, we denote the set of non-negative integers, with its natural order, by $\omega$.
\subsection{Ordered abelian groups and their automorphisms}

Recall that an \emph{ordered abelian group} (abbreviated as \emph{oag}) is an abelian group, which we will write additively, equipped with a linear order $<$ such that, for all $x,y,z$, if $x<y$ then $x+z<y+z$.

If $A$ is an oag and $a\in A$ we write, as usual, $\abs a$ for $a$ if $a\ge 0$ and $-a$ otherwise. An oag is \emph{Archimedean} iff, whenever $a,b\in A\setminus \set 0$, there is $n\in \omega$ such that $\abs a\le n  \abs b$. 

\begin{eg}
The oag $A\coloneqq (\mathbb Q + \sqrt 2 \mathbb Q, +, <)$ is Archimedean. The map $\sigma(x)\coloneqq \sqrt 2   x$ is an automorphism of $A$.
\end{eg}

A classical theorem of H\"older says that, for Archimedean oags, the example above is typical:
\begin{fact}[H\"older]\label{fact:holder}
  Up to isomorphism, the Archimedean oags are precisely the additive subgroups of $\mathbb R$, with the induced order. Moreover, every automorphism of such an oag is given by multiplication by a positive real number.
\end{fact}
A natural source of non-Archimedean oags is the following construction.
\begin{defin}Let $A$, $B$ be oags. The \emph{lexicographical product} $A\times_\mathrm{lex} B$ is the oag with underlying abelian group $A\times B$, ordered by setting $(a,b)>0$ iff either $a>0$, or $a=0$ and $b>0$. 
\end{defin}
One may iterate this construction, which is associative up to isomorphism. Versions with infinitely many factors are also possible, see \Cref{defin:hahnproduct} below.

If $A$ is an oag, its \emph{spine} is the set $I$ of its non-trivial principal convex subgroups, that is, convex hulls of non-trivial cyclic subgroups, ordered by reverse inclusion. The \emph{Archimedean valuation}\footnote{All valuation-theoretic notions in this paper, e.g.~skeleton, immediate extension, are with respect to the Archimedean valuation, the only valuation that we will use.} is the map $v\from A\to I\cup \set{\infty}$ sending $0$ to $\infty$ and every  $a\in A\setminus \set 0$ to the convex subgroup it generates. If $i\in I$, the \emph{rib} $A_i$ is the quotient $\set{a\in A: v(a)\ge i}/\set{a\in A: v(a)>i}$. When equipped with the order it inherits from $A$, it is a non-trivial Archimedean oag. The \emph{skeleton} $(I, (A_i: i\in I))$ of $A$ is given by its spine $I$ together with the $I$-sequence of its ribs. Note that, if $a,b\ne 0$, then  $v(a)>v(b)$ if and only if  $\abs a$ is much smaller than $\abs b$, that is, for all $n\in\omega$ we have $n  \abs a<\abs{b}$. In this case, we write $a\ll b$. We write $a\asymp b$ for $v(a)=v(b)$, and call $a/\asymp$ the \emph{Archimedean class} of $a$. We write  $a\sim b$ for $a-b\ll b$. It is (standard and) easily checked that $\sim$ is an equivalence relation on $A\setminus \set 0$. 

These notions also make sense for ordered vector spaces, that is, oags which are also equipped with the structure of a vector space over an ordered field $R$ in such a way  that scalar multiplication by positive $\lambda\in R$ preserves signs. In what follows, for simplicity, we adopt the convention below. We are mostly interested in the cases $R=\mathbb Q$ and $R=\mathbb R$, although a (positive) number of statements goes through in greater generality. 
 \begin{ass}
Except when otherwise stated, every ordered field $R$ in this paper will be Archimedean, and every ordered vector space will be over an Archimedean ordered field.
 \end{ass}
Ordered vector spaces over the reals are particularly well-behaved. For instance,  because ribs are non-trivial Archimedean oags, each of the ribs of an ordered $\mathbb R$-vector space must be isomorphic to $\mathbb R$. Moreover, we have the following easy consequence of the completeness of $\mathbb R$. 
 \begin{rem}\label{rem:asymptotics}
In every ordered $\mathbb R$-vector space, if $a,b\ne 0$, then $a\asymp b$ if and only if there is some $\lambda\in \mathbb R$ such that $a\sim \lambda  b$.
\end{rem}

\begin{defin}\label{defin:hahnproduct} Let $I$  be a linear order and $(A_i: i\in I)$ an $I$-sequence of non-trivial oags.
\begin{enumerate}
  \item The \emph{Hahn product}  $\operatorname{H}(I,(A_i: i\in I))$ is the oag with underlying abelian group the subgroup of $\prod_{i\in I} A_i$ given by those functions $a$ whose \emph{support} $\supp(a)\coloneqq\set{i\in I: a(i)\ne 0}$ is well-ordered, ordered by setting $a>0$ iff $a(\min (\supp(a)))>0$. 
  \item If all $A_i$ are equal to $B$, we write $B((I))$.
  \item The \emph{Hahn sum} $\coprod_{i\in I} A_i$ is the subgroup of $\operatorname{H}(I, (A_i: i\in I))$ given by those elements which have finite support, equipped with the induced order.
  \end{enumerate}
\end{defin}
\begin{eg}\label{eg:laurent}
If $I= \mathbb Z$ and each $A_i$ equals $\mathbb R$, then $\operatorname{H}(\mathbb Z, (\mathbb R : i \in \mathbb Z))=\mathbb R((\mathbb Z))$ is precisely the ordered group of (formal) real Laurent series in one variable $\mathbb R((t))$, ordered in such a way that $t<1$. Informally, they may be thought of as series expansions for $t\to 0$. The corresponding Hahn sum $\coprod_{i\in \mathbb Z} \mathbb R$ is the ordered subgroup $\mathbb R[t, t\inverse]$ of real Laurent polynomials in one variable.
\end{eg}
Power series notation is commonly used to denote elements of Hahn products (or Hahn sums): we write $\sum_{i\in I} a_i t^i$ for the function $a$, with $a_i\coloneqq a(i)$. If $a\ne 0$ and $i=\min\supp(a)$, we call $a_it^i$ the \emph{leading term} of $a$ and $a_i$ its \emph{leading coefficient}. Observe that the sign of $a$ equals the sign of its leading coefficient, and that if $a,b$ are non-zero elements of a Hahn product of non-trivial Archimedean oags  then $a\sim b$ if and only if $a$ and $b$ have the same leading term. As the notation suggests, if $(A_i:i\in I)$ is an $I$-sequence of non-trivial Archimedean oags, then the skeleton of the Hahn product $\operatorname{H}(I,(A_i: i\in I))$ is indeed $(I,(A_i: i\in I))$. These constructions too may be carried out for ordered vector spaces.

Recall that every oag $A$ embeds in its \emph{divisible hull}, which we may regard as the $\mathbb Q$-vector space $A\otimes_{\mathbb Z} \mathbb Q$, with the unique order extending that of $A$ and compatible with addition. Observe that this order is automatically preserved by multiplication by positive rationals. Moreover, every oag, or ordered vector space, embeds into the Hahn product of its skeleton, see~\cite[Proposition~2.3.4]{adh17}.   If $A$ is an oag with spine $I$, by first passing to its divisible hull, viewed as an ordered $\mathbb Q$-vector space, then embedding the latter into the Hahn product over its skeleton, and finally enlarging all the ribs to $\mathbb R$, we obtain an embedding of $A$ into $\mathbb R((I))$, and recover the classical Hahn Embedding Theorem. Observe that $\mathbb R((I))$ is naturally equipped with the structure of an ordered $\mathbb R$-vector space.

 An extension $A\subseteq B$ of oags, or of ordered vector spaces, is called \emph{immediate} iff $A$ and $B$ have the same skeleton, and $A$ is called \emph{maximally complete} iff it is non-trivial and has no proper immediate extensions. A maximally complete immediate extension of a non-trivial $A$ always exists: it suffices to take the Hahn product of its skeleton. Such an extension is unique, in the following sense.
\begin{fact}[{\cite[Theorem~0.26]{kuhlmann_ordered_2000}}]\label{fact:HET}
Let $\sigma\from A\to A'$ be an isomorphism of non-trivial ordered vector spaces, and let $A\subseteq B$ and $A'\subseteq B'$ be immediate extensions. If $B'$ is maximally complete, then $\sigma$ extends to an embedding $\sigma'\from B\to B'$. The map $\sigma'$ is an isomorphism if and only if $B$ is maximally complete.
\end{fact}

\begin{eg}\label{eg:laurentauto} 
The map $\sigma(\sum_{i\geq k} a_i t^i)= \sum_{i\geq k} 3a_i t^{i+5}$ is an automorphism of $\mathbb R((t))$, and it restricts to an automorphism of $\mathbb R[t,t\inverse]$.

More generally, given a linearly ordered set $I$ and oags $(A_i)_{i\in I}$,  if $g$ is an automorphism of $I$ and, for every $i\in I$,  we have an isomorphism of oags $g_i\from A_i\to A_{g(i)}$, then $\sigma\from  \operatorname{H}(I,(A_i: i\in I))\to 
\operatorname{H}(I,(A_i: i\in I))$ given by 
$\sigma(\sum_{i\in I} a_i t^i)= \sum_{i\in I} g_i(a_i)t^{g(i)}$ is an automorphism of the oag $\operatorname{H}(I,(A_i: i\in I))$, and restricts to an automorphism of $\coprod_{i\in I} A_i$. Not all automorphisms of Hahn products are of this form, e.g.~the one described in \Cref{eg:nolift} is not, and not all of them fix the Hahn sum subgroup setwise. See \cite{Kuhlmann_Serra_2022, ks_groups} for more on automorphisms of Hahn products.
\end{eg}

We conclude this subsection by recalling the following fact, which follows for instance from \cite[Corollaries~0.5 and~0.14, Proposition~0.25]{kuhlmann_ordered_2000}.
\begin{fact}\label{fact:ctblesep}
  Every countable dimensional ordered vector space is isomorphic to a Hahn sum of Archimedean ordered vector spaces.
\end{fact}

\subsection{A bit of model theory}
We briefly recall some standard facts from model theory. For the benefit of the reader not familiar with the subject, we begin with some examples, and refer such a reader to any standard introductory book, e.g.~\cite{hodges}, for precise definitions and more details.

The \emph{language of ordered abelian groups} is $L_\mathrm{oag}\coloneqq\set{+,0,-,<}$, where $0$ is a constant symbol, $+$ a binary function symbol, $-$ a unary function symbol, and $<$ a binary relation symbol. Each oag may be viewed as an \emph{$L_\mathrm{oag}$-structure} by \emph{interpreting} the symbols of $L_\mathrm{oag}$ in the way suggested by the notation. Every language contains a binary relation symbol $=$ by default, that must be interpreted in every structure as equality, that is, as the diagonal.

An example of a \emph{formula} in $L_\mathrm{oag}$ is $\phi(x,z)\coloneqq(\exists y\; x+z+0<y)\vee (x=z)$, where $\vee$ denotes (inclusive) disjunction; similarly, $\wedge$ denotes conjunction, and $\neg$ negation. The previous formula  has \emph{free variables} $x,z$, and the set of its solutions is a \emph{definable set}. A \emph{definable function} is a function whose graph is a definable set. A \emph{sentence} is a formula with no free variables. If $\phi$ is a sentence and $A$ a structure, we write e.g.~$A\models \phi$ to say that $\phi$  is \emph{satisfied} in $A$, that is, it is true in $A$. If for example $\phi(x)$ has a single free variable $x$, and $a\in A$, we write $A\models \phi(a)$, with the obvious meaning. If the ambient structure $A$ is clear from context, we also write $\models \phi(a)$, or $a\models \phi(x)$. Similar conventions apply when $x$ and $a$ are finite tuples.

A \emph{theory} $T$ is a set of sentences which is \emph{consistent}, that is, which has a \emph{model}: a structure where all $\phi\in T$ are satisfied; we write $A\models T$. For instance, the \emph{theory of oags} is the set of those $L_\mathrm{oag}$-sentences expressing the axioms of oags, e.g.~$\forall x,y,z\; ((x<y)\implica (x+z<y+z))$. An $L$-theory is \emph{complete} iff it is maximal amongst $L$-theories; equivalently, it is of the form $\set{\phi: A\models \phi}$ for some $L$-structure $A$. We will employ repeatedly the following standard fact; its use will usually be implicit, or announced by simply saying ``by compactness''.
\begin{fact}[Compactness Theorem]
  A set of sentences is consistent if and only if each of its finite subsets is consistent.
\end{fact}
An \emph{embedding} of $L$-structures $A\to B$ is an injective function commuting with the interpretations of the symbols of $L$. For example, the $L_\mathrm{oag}$-embeddings  between oags are precisely the strictly increasing group homomorphisms. In this paper, whenever we mention a category of structures, we implicitly mean that the arrows are embeddings in the appropriate language, unless otherwise stated.

A \emph{substructure} $A$ of $B$ is a subset of $B$, closed under the functions and constants of $L$, such that the inclusion map $A\to B$ is an embedding. We also say that $B$ is an \emph{extension} of $A$. If $A$, $B$ are $L$-structures, we write $A\subseteq B$ to mean that $A$ is a substructure of $B$.

If $A$ is an $L$-structure, we may expand $L$ to a language $L_A$ by adding a constant symbol $c_a$ for every $a\in A$, and make $A$ into an $L_A$-structure in the natural way. The \emph{diagram} $\operatorname{Diag}(A)$ is the set of all quantifier-free $L_A$-sentences which are true in $A$. It is easy to see that  $B\models\operatorname{Diag}(A)$ if and only if the map $A\to B$ sending $a$ to the interpretation of $c_a$ in $B$ is an embedding. We summarise this as follows for later reference.
\begin{fact}\label{fact:diagram}
  Models of $\operatorname{Diag}(A)$ may be identified with extensions of $A$.
\end{fact}

\begin{defin}
Let $R$ be an ordered field.
  \begin{enumerate}
  \item We set $L_R\coloneqq L_\mathrm{oag}\cup \set{\lambda\cdot-: \lambda\in R}$ and $L_{R, \sigma}\coloneqq L_R\cup \set{\sigma, \sigma\inverse}$, where $\sigma$, $\sigma\inverse$, and all of the $\lambda\cdot-$ are unary function symbols.
  \item Let $R\mathsf{-OVS}$ be the $L_R$-theory of non-trivial ordered $R$-vector spaces.
  \item  Let $R\mathsf{-OVSA}$ be the $L_{R,\sigma}$-theory of (possibly trivial) ordered $R$-vector spaces with an $L_R$-automorphism\footnote{That is, an $R$-linear strictly increasing bijection.} $\sigma$ and its inverse $\sigma\inverse$.
  \end{enumerate}
\end{defin}
Note that an $L_{R, \sigma}$-embedding is the same as an $R$-linear strictly increasing map commuting with $\sigma$ (and $\sigma\inverse$).

The theory $R\mathsf{-OVS}$ is complete and \emph{o-minimal}, see \cite[Chapter 1, Corollary~7.6]{van_den_dries_tame_1998}. 
Besides a somehow implicit use in \Cref{thm:ivp}, we will only exploit o-minimality once by invoking the \emph{Cell Decomposition Theorem} in the proof of \Cref{lemma:trans_AP}; for this reason, we will not define it, and refer the reader to the aforementioned source.

A key role in what follows will be played by the notion of an \emph{existentially closed} structure. We recall its definition, specialised to the current setting.
\begin{defin}\label{defin:ec}
We call $M\models R\mathsf{-OVSA}$ \emph{existentially closed} iff, for every quantifier-free $L_{R,\sigma}$-formula $\phi(\bla x0,n, \bla y0,k)$ and each tuple  $(\bla a0,n)$ from $M$, if there is some $B\models R\mathsf{-OVSA}$ with $B\supseteq M$ containing $\bla b0,k$ such that  $B\models \phi(\bla a0,n, \bla b0,k)$, then there are $\bla m0,k\in M$ such that  $M\models  \phi(\bla a0,n, \bla m0,k)$.
\end{defin}
The following is a standard consequence of the fact that the category of models of $R\mathsf{-OVSA}$ with embeddings is closed under direct limits or, equivalently, that $R\mathsf{-OVSA}$ can be axiomatised by $\forall\exists$-sentences, see~\cite[Section~8.2]{hodges}.
\begin{fact}
  Every model of $R\mathsf{-OVSA}$ embeds into an existentially closed one.
\end{fact}
The following remark will give us a first glimpse of these structures. Its proof is a typical example of how one shows that existentially closed structures satisfy certain properties: it suffices to build larger structures that do.
\begin{rem}\label{rem:aspfp}
Every existentially closed $M\models \mathbb R\mathsf{-OVSA}$ contains arbitrarily small positive fixed points of $\sigma$.
\end{rem}
\begin{proof}
Consider the formula $\phi(x,y)\coloneqq (0<y<x)\wedge (\sigma(y)=y)$.  We need to show that, for every positive $a\in M$, there is $m\in M$ with $M\models \phi(a,m)$. By \Cref{defin:ec}, it suffices to build an extension $B$ of $M$ containing a fixed point between $0$ and $a$. This is easily done by taking as $B$ the oag $M\times_\mathrm{lex} \mathbb R$ together with the unique extension of $\sigma$ which is the identity on  $\set{0}\times\mathbb R$.
\end{proof}
 Analogously, one shows that $M$ contains arbitrarily large fixed points, arbitrarily small positive $x$ such that $\sigma(x)\ll x$, or such that $\sigma(x)=\sqrt{7}  x$, etc. Similarly, it can be shown that no interval $[a,b]\subseteq M$ may consist entirely of fixed points: it suffices to take $M\times_\mathrm{lex}\mathbb R$, make $\sigma$ act non-identically on $\set{0}\times\mathbb R$, and consider $(a,\epsilon)$ for an arbitrary positive $\epsilon\in \mathbb R$.

\section{Realification}\label{sec:realification}
Working with (automorphisms of) ordered real vector spaces, instead of arbitrary oags, will allow us to exploit a number of desirable properties. For instance, we will frequently use the following consequence of \Cref{rem:asymptotics}.
\begin{rem}\label{rem:asymptsigma}
For every $a\in A\models \mathbb R\mathsf{-OVSA}$, either $\sigma(a)\gg a$, or $\sigma(a)\ll a$, or there is a positive real $\rho$ such that $\sigma(a)\sim \rho  a$. Moreover, if $\sigma(a)\sim \rho  a$, then, for every $b$ in the convex hull of the $\sigma$-orbit of $a$, we also have $\sigma(b)\sim \rho  a$, and analogously if $\sigma(a)\ll a$ or $\sigma(a)\gg a$. 
\end{rem}

\begin{eg} 
Let $F\supseteq \mathbb{R}$ be a non-archimedean ordered field extending the field of real numbers. Note that $F$ admits a natural structure of an ordered $\mathbb{R}$-vector space.   Let $a\in F$ with $a>0$. Define $\sigma$ on $F$ by $\sigma(x)\coloneqq ax$ for every $x\in F$. Then $\sigma$ is an automorphism of the ordered $\mathbb{R}$-vector space $F$ and:
 \begin{itemize}
 \item if $a>\mathbb R$ then $\sigma(x)\gg x$ for every $x\neq 0$, 
\item if $a$ is a positive infinitesimal (i.e.~$0<a<1/n$ for every $n\in \omega\setminus \{0\}$),  then $\sigma(x)\ll x$ for every $x\neq 0$,
 \item if none of the above holds, then there is a \emph{standard part} $\rho$ of $a$, that is, a unique real number $\rho$ for which $\abs{a-\rho}$ is an infinitesimal. Such a $\rho$ must be positive because $a$ is, and  $\sigma(x)\sim \rho  x$ for every $x$. 
\end{itemize}
 \end{eg} 

Motivated by \Cref{rem:asymptsigma}, in this section we reduce an amalgamation problem of oags with an automorphism to one of ordered $\mathbb R$-vector spaces with an automorphism.  The first, easy step, consists in assuming all of our oags to be divisible, by passing to their divisible hull.  Elements of the divisible hull of $A$ are of the form $a/n$, for suitable $a\in A$ and $n>0$, and it follows easily that the divisible hull is unique up to \emph{unique} isomorphism over $A$. This has the consequence below.
\begin{rem}\label{rem:oagtoqovs}
Every automorphism $\sigma$ of $A$ extends uniquely to an automorphism $\sigma'$ of its divisible hull. The latter is naturally equipped with the structure of an ordered $\mathbb Q$-vector space, and this structure is preserved by $\sigma'$.
\end{rem}
This, in turn, is easily seen to imply the following.
\begin{rem}\label{rem:oagstoqovs}
Every amalgamation problem of oags with an automorphism reduces to one of ordered $\mathbb Q$-vector spaces with an ordered $\mathbb Q$-vector space automorphism.
\end{rem}

Suppose now that $A$ is an ordered $\mathbb Q$-vector space, with an automorphism $\sigma_A$. By the Hahn Embedding Theorem, if $I$ is the spine of $A$, we may embed $A$ into the Hahn group $B\coloneqq\mathbb R((I))$. It is also possible to show\footnote{For example by using~\cite[Proposition~13.12]{hhm}.} that $\sigma_A$ can be lifted to an automorphism $\sigma_B$ of $B$ \emph{as an ordered $\mathbb Q$-vector space}.  Unfortunately, as soon as $I$ has size at least $2$, it may happen that no lift $\sigma_B$ preserves the $\mathbb R$-vector space structure.
\begin{eg}\label{eg:nolift}
  Let $A\coloneqq ((\mathbb Q + \sqrt 2 \mathbb Q) \times_\mathrm{lex} \mathbb Q)$, equipped with the automorphism $\sigma_A((a+\sqrt 2 b,c))=(a+\sqrt 2 b, c+b)$. Since $A$ has two Archimedean classes, the Hahn Embedding Theorem gives us an embedding  of ordered $\mathbb Q$-vector spaces $\iota\from A\to B\coloneqq\mathbb R\times_\mathrm{lex}  \mathbb R$, and a lift $\sigma_B$ of $\sigma_A$ to $B$. For every $\iota$ and $\sigma_B$ as above, the latter is not an isomorphism of $\mathbb R$-vector spaces.
\end{eg}
\begin{proof}
Towards a contradiction, assume $\iota$ and $\sigma_B$ are given, with the former commuting with $\sigma_-$ and the latter with multiplication by every real number. It is routine to check that we must have $\sqrt 2\cdot \iota(1,0)\sim \iota(\sqrt 2, 0)$, but that the two cannot be equal, since the former must be a fixed point of $\sigma_B$, but the latter is not. Since there are only two Archimedean classes in $B$, there must be $d\in \mathbb R$ with $\sqrt 2\cdot \iota(1,0)-\iota(\sqrt 2, 0)= (0, d)$. Again because there are only two Archimedean classes, there must be a positive $e\in \mathbb R$ with $\iota(0,1)=(0,e)$. Since this is a fixed point of $\sigma_B$, which is an isomorphism of real vector spaces, so is $(0,d)=\frac de\cdot (0,e)$. But then the non-fixed point $\iota(\sqrt 2, 0)$ is the difference of two fixed points, namely  $\sqrt 2\cdot \iota(1,0)$ and $(0, d)$, a contradiction.
\end{proof}
In the proof above, we used that $B$ had only two Archimedean classes. This was in fact crucial: we leave it to the reader to check that if we embed $A$ into the lexicographically ordered $\mathbb R^3$ as $(a+\sqrt 2b,c)\mapsto (a+\sqrt 2b,b,c)$ then we may lift $\sigma_A$ to an ordered $\mathbb R$-vector space automorphism. In fact, adding extra Archimedean classes, together with a dash of compactness, allows us to prove the following general statement. 
\begin{thm}\label{thm:redtoR}
  Let $A$ be an oag. There is an ordered $\mathbb R$-vector space $B$ and an embedding of oags $A\to B$ such that every automorphism of $A$ as an oag extends to an automorphism of $B$ as an ordered $\mathbb R$-vector space. Moreover, this can be done in such a way that the map sending each automorphism to its extension is an embedding of groups $\aut(A)\to \aut_{\mathbb R}(B)$.
\end{thm}
\begin{proof}
  Let $L$ be the language of oags, together with a unary function symbol for every $\sigma\in \aut(A)$ and a constant symbol for every $a \in A$.  View $A$ as an $L$-structure in the natural way. By \Cref{fact:diagram}, in order to prove the first part of the theorem it suffices to show that $\operatorname{Diag}(A)\cup \mathbb R\mathsf{-OVS}$, together with the (infinitely many) sentences expressing that each $\sigma$ is an ordered $\mathbb R$-vector space automorphism, is consistent. In order to also prove the ``moreover'' part, we enlarge this set of sentences as follows. For every $\sigma_0,\sigma_1,\sigma_2\in \aut(A)$ such that $\sigma_0\sigma_1=\sigma_2$, we add the sentence $\forall x\; \sigma_0(\sigma_1(x))=\sigma_2(x)$.

 By compactness, we reduce to proving consistency of a finite subset of the above, which may only mention a finite set, say $\Sigma_0$, of automorphism  symbols, and a finite set $A'\subseteq A$. Let $L'$ be the expansion of $L_{\mathrm{oag}}$ by the symbols in $\Sigma_0$. The divisible hull of the $L'$-structure generated by $A'$ is a $\mathbb Q$-vector space of at most countable dimension, and our automorphisms extend uniquely to it by \Cref{rem:oagtoqovs}. Using \Cref{fact:ctblesep}, we reduce to the case where $A$ is a Hahn sum $\coprod((A_i)_{i\in I}, I)$, where $I$ is a countable linear order and every $A_i$ is an Archimedean ordered $\mathbb Q$-vector space, equipped with a finite set $\Sigma_0$ of automorphisms. For each $i$, fix an embedding of $A_i$ in $\mathbb R$, identify $A_i$ with its image under this embedding, and do so in such a way that if $A_i\cong A_{i'}$ then $A_i=A_{i'}$.

    By~\cite[Corollary~3.21]{ks_groups}\footnote{This is stated for automorphisms which are only required to preserve the valuation, not necessarily the order. Nevertheless, since internal automorphisms preserve leading terms, they preserve the order, hence $\sigma_\mathrm{i}$ is an automorphism of oags  (hence automatically of ordered $\mathbb Q$-vector spaces) and it follows that so is $\sigma_\mathrm{e}=\sigma_\mathrm{i}\inverse\sigma$.} we may decompose each $\sigma\in \Sigma$ as $\sigma=\sigma_\mathrm{i}\sigma_\mathrm{e}$, where $\sigma_\mathrm{i}$ is \emph{internal}, that is, it induces the identity on the skeleton $(I, (A_i: i\in I))$ of $A$,  or in other words preserves the leading term of each $\sum_{i\in I} a_i t^i\in A$, and 
$\sigma_\mathrm{e}$ is \emph{external}, i.e.~of the form $\sigma_\mathrm{e}(\sum_{i\in I} a_i t^i)= \sum_{i\in I} g_i(a_i)t^{g(i)}$, for suitable isomorphisms $g\from I\to I$ and $g_i\from A_i\to A_{g(i)}$.
  The maps $g_i\from A_i\to A_{g(i)}$ are isomorphisms of Archimedean ordered $\mathbb Q$-vector spaces, and by construction $A_i=A_{g(i)}\subseteq \mathbb R$.  By \Cref{fact:holder} we may identify each $g_i$ with multiplication by a positive real number $\rho_i$, hence each external $\sigma$ has the form\footnote{We encountered such an automorphism in \Cref{eg:laurentauto}.}
  \begin{equation}
    \label{eq:genericexternalsigma}
    \sigma\left(\sum_{i\in I} a_i t^i\right)= \sum_{i\in I} \rho_ia_it^{g(i)}
  \end{equation}
 By enlarging $\Sigma_0$ if necessary, we may assume that whenever it contains $\sigma$ it also contains some $\sigma_\mathrm{i}$ and $\sigma_\mathrm{e}$ as above.

  Let $J$ be a dense linear order with no endpoints extending $I$ which is \emph{$\aleph_1$-saturated}: that is, such that every countable sequence of intervals with the finite intersection property has non-empty intersection.  We define an $L_\mathrm{\mathbb Q, \sigma}$-embedding $f\from A\to \mathbb R((J))$ as follows.  Recall that, for every $i\in I$, the dimension  $\dim_{\mathbb Q}(A_i)$ of $A_i$ as a $\mathbb Q$-vector space is either finite or $\aleph_0$. For each $i\in I$, choose a basis $\set{b_{i,j}: j<\dim_{\mathbb Q}(A_i)}$ of $A_i$ as a $\mathbb Q$-vector space, making sure to choose the same basis \emph{indexed in the same order} whenever $A_i=A_{i'}$. For each $i\in I$ use $\aleph_1$-saturation of $J$ to find an increasing sequence $(e_{i,j})_{j<\dim_{\mathbb Q}(A_i)}$ such that $e_{i,j}\models \set{i<x<i': i'\in I, i<i'}$, and set, in power series notation, $\epsilon_{i,j}\coloneqq t^{e_{i,j}}$. Define an embedding $f\from A\to \mathbb R((J))$ as the $\mathbb Q$-linear extension of the map $b_{i,j}t^i\mapsto b_{i,j}t^i+\epsilon_{i,j}$. It is routine to verify that the  map $f$ is an embedding of $\mathbb Q$-vector spaces. Moreover, it maps positive elements to positive elements, as is easily checked by observing that the leading terms of $a$ and of $f(a)$ are the same. Therefore, $f$ is an embedding of ordered $\mathbb Q$-vector spaces. 
Let $B$ be the real vector space generated by $f(A)$ in $\mathbb R((J))$.

For each $\sigma\in \Sigma_0$, we now extend $f\sigma f\inverse$  from $f(A)$ to an automorphism $\tilde \sigma$ of $B$ as an ordered $\mathbb R$-vector space. For notational simplicity, if $c\in f(A)$, write $\sigma(c)$ in place of $f(\sigma(f\inverse(c)))$. Every element of $B$ may be written as
\begin{equation}
  b=\sum_{\substack{i\in I\\ j<\dim_{\mathbb Q}(A_i)}} \lambda_{i,j}(b_{i,j}t^i+\epsilon_{i,j})\label{eq:genericelementofB}
\end{equation}
where only finitely many $\lambda_{i,j}$ are non-zero and  $\lambda_{i,j}\in \mathbb R$. The presence of the $\epsilon_{i,j}$, ensures that such a writing is unique. In other words, $f$ induces an isomorphism of $\mathbb R$-vector spaces $A\otimes_{\mathbb Q} \mathbb R\to B$. Since $\mathbb Q$ is a field, the functor $-\otimes_{\mathbb Q} \mathbb R$ is exact, yielding a natural extension $\tilde\sigma$ of $\sigma$ to an automorphism of the $\mathbb R$-vector space $A\otimes_{\mathbb Q} \mathbb R$, henceforth identified with $B$. After this identification, for $b$ as in~\eqref{eq:genericelementofB}, we have the concrete description
\begin{equation}\label{eq:generictsigmab}
  \tilde\sigma(b)=\sum_{\substack{i\in I\\ j<\dim_{\mathbb Q}(A_i)}} \lambda_{i,j}(\sigma(b_{i,j}t^i+\epsilon_{i,j}))
\end{equation}
From this, it follows immediately that, for every $\sigma\in \Sigma_0$, we automatically have $\tilde\sigma_\mathrm{i}\tilde\sigma_\mathrm{e}=\tilde \sigma$. Similarly,
 whenever $\sigma_0\sigma_1=\sigma_2$, we automatically have $\tilde\sigma_0\tilde\sigma_1=\tilde\sigma_2$.

 To conclude, we need to check that, for every $\sigma\in \Sigma_0$, the $\mathbb R$-linear bijection $\tilde\sigma$ preserves the order. Because we added $\sigma_\mathrm{i}, \sigma_\mathrm{e}$ to $\Sigma_0$ and checked that $\tilde\sigma_\mathrm{i}\tilde\sigma_\mathrm{e}=\tilde \sigma$, it suffices to prove this in the special cases where $\sigma$ is internal or external.

Assume $\sigma$ is internal, write a generic $b\in B$ as in~\eqref{eq:genericelementofB}, and recall that, if $b\ne 0$, then its sign is decided by its leading coefficient. Let $i_0$ be the minimum $i$ such that some $\lambda_{i,j}$ is non-zero, and let $j_0$ be minimum such that $\lambda_{i_0, j_0}$ is non-zero. We have two cases.
\begin{enumerate}
\item If $\sum_{j<\dim_{\mathbb Q}(A_{i_0})} \lambda_{i_0,j} b_{i_0, j}\ne 0$, then  $\left(\sum_{j<\dim_{\mathbb Q}(A_{i_0})} \lambda_{i_0,j} b_{i_0, j}\right) t^{i_0}$ is the leading term of  $b$. Since $\sigma$ is internal, it follows from~\eqref{eq:generictsigmab} that this is also the leading term of $\tilde \sigma(b)$, which therefore has the same sign as $b$.
\item Otherwise,  the leading term of $b$ is $\lambda_{i_0, j_0} \epsilon_{i_0, j_0}$. Again because $\sigma$ is internal, and by choice of the $\epsilon_{i,j}$, a straightforward calculation shows that this is also the leading term of $\tilde \sigma(b)$, which, again, has the same sign as $b$.
\end{enumerate}
Therefore, in the case where $\sigma\in \Sigma_0$ is internal,  $\tilde\sigma$ is an ordered $\mathbb R$-vector space automorphism of $B$ extending $\sigma$.

Assume now that $\sigma$ is external and write it as in \eqref{eq:genericexternalsigma}. Then, if $b$ is as in~\eqref{eq:genericelementofB}, by combining \eqref{eq:genericexternalsigma} with \eqref{eq:generictsigmab},
\[
  \tilde\sigma(b)= \sum_{\substack{i\in I\\ j<\dim_{\mathbb Q}(A_i)}} \lambda_{i,j}(\rho_i(b_{i,j}t^{g(i)}+\epsilon_{g(i),j}))
\]
The argument to show that $\tilde\sigma$ preserves the order is similar to the internal case: again, let $i_0$ be the minimum $i$ such that some $\lambda_{i,j}$ is non-zero, and let $j_0$ be minimum such that $\lambda_{i_0, j_0}$ is non-zero.
\begin{enumerate}
\item If $\sum_{j<\dim_{\mathbb Q}(A_{i_0})} \lambda_{i_0,j} b_{i_0, j}\ne 0$, then $\left(\sum_{j<\dim_{\mathbb Q}(A_{i_0})} \lambda_{i_0,j} b_{i_0, j}\right) t^{i_0}$ is the leading term of $b$. It follows immediately that the leading term of $\tilde \sigma(b)$ equals $\left(\rho_{i_0} \sum_{j<\dim_{\mathbb Q}(A_{i_0})} \lambda_{i_0,j} b_{i_0, j}\right) t^{g(i_0)}$, and we conclude since $\rho_{i_0}$ is positive.
\item Otherwise, the leading term of $b$ is $\lambda_{i_0, j_0} \epsilon_{i_0, j_0}$. 
Since in our construction we used the same basis for $A_{i_0}$ and $A_{g(i_0)}$, indexed in the same order, it follows that the map $\epsilon_{i_0,j}\mapsto \epsilon_{g(i_0), j}$ is increasing, and therefore the leading term of $\tilde \sigma(b)$ is $\rho_{i_0}\lambda_{i_0, j_0} \epsilon_{g(i_0), j_0}$. Again, we conclude since $\rho_{i_0}$ is positive.
\end{enumerate}
This shows that, when $\sigma\in \Sigma_0$ is external,  $\tilde\sigma$ is an ordered $\mathbb R$-vector space automorphism of $B$, thereby concluding the proof.
\end{proof}

\begin{co}\label{co:redrovsa}
    Let $A\models \mathbb Q\mathsf{-OVSA}$. Then there are $B\models \mathbb R\mathsf{-OVSA}$ and an $L_{\mathbb Q,\sigma}$-embedding $A\to B$.
\end{co}

\begin{pr}\label{pr:Rreduction}
 If the category of models of $\mathbb R\mathsf{-OVSA}$ with maps the $L_{\mathbb R, \sigma}$-embeddings has the AP, then so does the category of models of $\mathbb Q\mathsf{-OVSA}$ with maps the $L_{\mathbb Q, \sigma}$-embeddings.
\end{pr}
\begin{proof}
	\color{blue}
 Given an amalgamation problem $C\leftarrow A\to B$ of models of $\mathbb Q\mathsf{-OVSA}$, we produce an  amalgamation problem $C'\leftarrow A'\to B'$ of models of $\mathbb R\mathsf{-OVSA}$, together $L_{\mathbb Q,\sigma}$-embeddings $A\to A'$, $B\to B'$ and $C\to C'$ commuting with those in said amalgamation problems.

Take the reduct of $C\leftarrow A\to B$ to $L_{\mathbb Q}$; since $\mathbb Q\mathsf{-OVS}$ eliminates quantifiers, there is a solution $D_0$ of the resulting amalgamation problem. Let $D_1$ be a $\abs{D_0}^+$-strongly homogeneous elementary extension of $D_0$, and identify the $L_{\mathbb Q}$-reducts of $A,B,C$ with the corresponding substructures of $D_1$.

Again by quantifier elimination, $\sigma_B$ is a partial elementary map from $D_1$ to itself. By strong homogeneity, we extend $\sigma_B$ to an automorphisms $\sigma_{B_1}$ of $D_1$, and set $B_1\coloneqq(D_1,\sigma_{B_1})$. By running the same argument with $B$ replaced by $C$, we analogously obtain a structure $C_1\coloneqq(D_1,\sigma_{C_1})$, on the same underlying ordered $\mathbb Q$-vector space $D_1$.

Apply Theorem \ref{thm:redtoR} to $D_1$, obtaining a real vector space $D_2$ extending $D_1$, together with extensions of every automorphism of $D_1$ to one of $D_2$. Let $\sigma_{B_2}$ and $\sigma_{C_2}$ be the resulting extensions of  $\sigma_{B_1}$ and $\sigma_{C_1}$ respectively.

Observe that $\sigma_{B_2}$ and $\sigma_{C_2}$ both restrict to $\sigma_A$ on $A$, hence they have the same restriction  $\sigma_{A'}$ to the real vector space $A'$ it generates. By construction, $\sigma_{A'}\left(\sum_i r_i a_i\right)=\sum_i r_i \sigma_A(a_i)$, hence $\sigma_{A'}(A')\subseteq A'$. 
Similarly, as $\sigma_{B_2}\inverse$ restricts to an automorphism of $A$, and is real vector space automorphism, we have $\sigma_{B_2}\inverse(A')\subseteq A'$. It follows that the restriction of $\sigma_{B_2}\inverse$ to $A'$ is the compositional inverse of $\sigma_{A'}$, which is therefore an automorphism of $A'$.

We may therefore conclude by considering $(A',\sigma_{A'})$, together with its inclusion in  $B'\coloneqq (D_2, \sigma_{B_2})$ and $C'\coloneqq (D_2, \sigma_{C_2})$, and the embeddings $A\to A'$, $B\to B'$, $C\to C'$ given by the construction above.
\end{proof}

\section{Absolute monotonicity}\label{sec:OVSA}
We are interested in the behaviour of \emph{$\sigma$-polynomials}, that is,  linear combinations of iterates of $\sigma$, on existentially closed structures. For simplicity, we only deal with the case $R=\mathbb R$, even if some statements and proofs would work in broader generality.
\begin{defin}\label{defin:tilde} \*
  \begin{enumerate}
\item A \emph{$\sigma$-polynomial} is an element of the polynomial ring $\mathbb R[\sigma]$. 
  \item We identify $\sum_i \lambda_i \sigma^i$ with the induced definable function $x\mapsto\sum_i \lambda_i \sigma^i(x)$. 
  \item If we want to think of $f\in \mathbb R[\sigma]$ as an actual polynomial, instead of the definable function it induces, then we write $\tilde f\in \mathbb R[y]$, and call $\tilde f$ the polynomial \emph{associated} to $f$.
  \item If we want to stress the distinction between $f\in \mathbb R[\sigma]$ and the definable function it induces on a specific $A\models R\mathsf{-OVSA}$, we write $f(x)\from A\to A$ for the latter.
  \item If $A\models \mathbb R\mathsf{-OVSA}$, by a \emph{$\sigma$-polynomial over $A$} we mean a definable function of the form $f(x)=g(x)+d$, with $g\in \mathbb R[\sigma]$ and $d\in A$. We write $f^*$ for $g$, and $\tilde f$ for $\tilde g$.
  \end{enumerate}
\end{defin}

\begin{rem}
 One may also consider a more general notion of $\sigma$-polynomial, by allowing $\sigma\inverse$ to be mentioned, namely, one can look at $\mathbb R[\sigma, \sigma\inverse]$. Since every such object can be written as $\sigma^{-n}f$, for suitable $n\in \omega$ and $f\in \mathbb R[\sigma]$, most statements in what follows generalise immediately from $\mathbb R[\sigma]$ to $\mathbb R[\sigma, \sigma\inverse]$.
\end{rem}
\subsection{Positive real roots}
\Cref{rem:asymptsigma}, together with a routine calculation, gives us the following property, which will be used repeatedly.
\begin{rem}\label{rem:popcwti}
Suppose that $f$ is a $\sigma$-polynomial, $a\in A\models \mathbb R\mathsf{-OVSA}$ is non-zero, and $\rho\in \mathbb R$ is such that $\sigma(a)\sim \rho  a$. Then $\rho>0$ and, if $\tilde f(\rho)\ne 0$, then $f(a)\sim \tilde f(\rho)  a$.
\end{rem}
  \begin{lemma}\label{lemma:ubddimpec}
  If $f\in  \mathbb R[\sigma]\setminus \set 0$  and $M\models \mathbb R\mathsf{-OVSA}$ is existentially closed, then $f(x)\from M\to M$ has unbounded image.
\end{lemma}
\begin{proof}
It suffices to show that, if $f(x)$ has the form $\sigma^n(x)+f_0(x)$, with $f_0$ of degree less than $n$, then $\im f$ has no upper bound. Suppose that $d\in M$ is such that $\im f\le d$. We reach a contradiction by building an extension $B\supseteq M$ containing some $b$ such that $f(b)>M$. This is readily done by  taking the o-minimal \emph{prime model} $B= M(e_i: i\in\mathbb Z)$, where $e_i$ is a \emph{Morley sequence} in the $L_{\mathbb R}$-type over $M$ at $+\infty$. Concretely, $B$ is the $\mathbb R$-vector space generated by $M$ and the $e_i$, with the unique compatible order where the $e_i$ are all positive and $M\ll e_i\ll e_{i+1}$.
Then, extend $\sigma$ by letting it act on $(e_i)_{i\in \mathbb Z}$ by shift.
\end{proof}

Below, we introduce the notion of an \emph{absolutely monotone} $\sigma$-polynomial. The name is inspired by absolutely irreducible algebraic varieties.
\begin{defin}
  Let $f\in \mathbb R[\sigma]$ be a $\sigma$-polynomial.
  \begin{enumerate}
  \item We say that $f$ is \emph{absolutely monotone} iff $\tilde f\ne 0$ and, for every $A\models \mathbb R\mathsf{-OVSA}$, the map $f\from A\to A$ is monotone.
  \item We similarly define \emph{absolutely increasing} and \emph{absolutely decreasing}.
  \item A $\sigma$-polynomial $f$ over $A$ is \emph{absolutely monotone}/\emph{increasing}/\emph{decreasing} iff $f^*$ is.
  \end{enumerate}
\end{defin}

\begin{rem}\*\label{rem:amsscoi}
  \begin{enumerate}
  \item\label{point:absmononesol} If $f$ is absolutely monotone then, for every
    $A\models \mathbb R\mathsf{-OVSA}$, and every $d\in A$, the equation
    $f(x)=d$ has at most one solution.
  \item Absolute monotonicity may be checked on an interval. More precisely, if $[a,b]\subseteq A\models \mathbb R\mathsf{-OVSA}$ (with $a<b$) and $f$ is a $\sigma$-polynomial over $A$ which, in every $B\models \mathbb R\mathsf{-OVSA}$ extending $A$, has a monotone restriction to $[a,b]$, then $f$ is absolutely monotone. This follows from the fact that, if $f$ is not monotone on $C\models \mathbb R\mathsf{-OVSA}$, then in $B\coloneqq A\times_\mathrm{lex} C$ it will also fail to be monotone on every $[a,b]$ as above.
  \end{enumerate}
\end{rem}

\begin{pr}\label{pr:famighnprp}
For every $f\in \mathbb R[\sigma]\setminus \set 0$  the following are equivalent.
\begin{enumerate}
\item\label{mon1} The $\sigma$-polynomial $f$ is absolutely monotone.
\item\label{mon2} For every existentially closed $M$, the function $f\from M\to M$ is monotone.
\item\label{mon3} The polynomial $\tilde f(y)$ has no positive real root.
\item\label{mon4} The $\sigma$-polynomial $f$ is absolutely increasing or absolutely decreasing.
\end{enumerate}
Moreover, if the above hold, then we also have the following.
\begin{enumerate}[label=(\alph*)]
\item \label{point:absmonmoreovera}Let $f=\sum_{i\le \deg f} \lambda_i \sigma^i$, and let $\bar \imath$ be the minimum such that $\lambda_{\bar\imath}\ne 0$. Then, $\lambda_{\bar\imath}$ and the leading coefficient $\lambda_{\deg f}$ have the same sign. They are both positive if and only if $f$ is absolutely increasing, and both negative if and only if $f$ is absolutely decreasing.
\item\label{point:absmonmoreoverb} If $a\in A\setminus\set0$ is such that $\sigma(a)\asymp a$, then $f(a)\asymp a$. More precisely, if $\sigma(a)\sim \rho   a$, then $f(a)\sim \tilde f(\rho )  a$.
\end{enumerate}
\end{pr}
\begin{proof}
Both $\ref{mon4}\allora \ref{mon1}$ and  $\ref{mon1}\allora \ref{mon2}$ are trivial.

  For $\ref{mon2}\allora \ref{mon3}$, suppose that $\rho>0$ is a real root of $\tilde f(y)$. By \Cref{lemma:ubddimpec}, to contradict monotonicity it suffices to show that $f$, which is non-zero, has arbitrarily large zeroes (hence, by considering their additive inverses, also arbitrarily small ones). To this end, we take the lexicographical product $\mathbb R\times_\mathrm{lex} M$.  Since $\rho>0$, we may extend $\sigma$ to $ \mathbb R\times_\mathrm{lex}M$ by mapping $(x,y)\mapsto (\rho  x, \sigma(y))$. It follows that $(1,0)$ is a zero of $f$ larger than $M$, and since $M$ is existentially closed this implies that $f$ has arbitrarily large zeroes in $M$.

  For $\ref{mon3}\allora \ref{mon4}$, suppose that $f=\sum_{i\le \deg f}\lambda_i\sigma^i$. By composing with $\sigma^{-k}$ for  suitable $k\in \omega$, we may assume that $\lambda_0\ne 0$, i.e., in the notation of \ref{point:absmonmoreovera}, that $\bar\imath=0$. Clearly, if $\tilde f(y)=\sum_{i\le \deg f} \lambda_i y^i$ is the associated polynomial, then $\tilde f(0)=\lambda_0$.   Since $\tilde f$ has no positive real root, it must have constant sign on the positive reals. In particular, since $\lim_{y\to \infty} \tilde f(y)=\sign(\lambda_{\deg f})\cdot (+\infty)$, it follows that $\lambda_0$ and $\lambda_{\deg f}$ must have the same sign, incidentally proving the first part of~\ref{point:absmonmoreovera}. We prove that, if both are positive, then $f$ is absolutely increasing; clearly, if both are negative, we obtain that $f$ is absolutely decreasing by applying the positive case to $-f$. This proves $\ref{mon4}$ which, together with the fact that $\lambda_0$, $\lambda_{\deg f}$ have the same sign, proves the rest of~\ref{point:absmonmoreovera}.

Fix $A\models \mathbb R\mathsf{-OVSA}$. Since $f(x)$ is linear, it is enough to show that $A\models \forall x\; (x>0\implica f(x)> 0)$. We fix a positive $a\in A$ and distinguish three cases, depending on comparability of $a$ and $\sigma(a)$ according to the Archimedean valuation. If $\sigma(a)\gg a$, then $\sign(f(a))=\sign(\lambda_{\deg f})$, and if $\sigma(a)\ll a$, then $\sign(f(a))=\sign(\lambda_{0})$. Since we are assuming both $\lambda_{\deg f}$ and $\lambda_{0}$ are positive, in both of these cases we are done. In the only remaining case there is a positive real number $\rho$ such that $\sigma(a)\sim \rho  a$. By \Cref{rem:popcwti} we have  $f(a)\sim\tilde f(\rho)  a$. Since $\tilde f(0)>0$ and, as observed above, $\tilde f$ has constant sign on the positive reals, we obtain $\tilde f(\rho)>0$ and we are done; this also proves~\ref{point:absmonmoreoverb}.
\end{proof}

\begin{co}
  If $f\in \mathbb R[\sigma]$ is absolutely monotone, and $A\models \mathbb R\mathsf{-OVSA}$, then $f(x)\from A\to A$ has unbounded image.
\end{co}
\begin{proof}
  It is enough to prove unboundedness from above when $f$ is absolutely increasing, and moreover with non-zero identity term, that is, $f=\sum_{i\le \deg f} \lambda_i\sigma^i$ with $\lambda_0> 0$.  Suppose that for such an $f$ we have  $\im f < a\in A$. By \Cref{pr:famighnprp}, if $\sigma(a)\gg a$ then $f(a)\asymp \sigma^{\deg f}(a)\gg a$, and if  $\sigma(a)\ll a$ or   $\sigma(a)\asymp a$ then  $f(a)\asymp a$. Hence, using \ref{point:absmonmoreovera} and \ref{point:absmonmoreoverb} from \Cref{pr:famighnprp}, for some $\beta\in \mathbb Z$ we have $\beta  f(a)> a$, but clearly $\beta  f(a)=f(\beta  a)$, a contradiction.
\end{proof}
\subsection{Surjectivity}
We adopt the following conventions.
\begin{defin} Let $A\subseteq B$ be linear orders.
  \begin{enumerate}
  \item A \emph{cut} $p$ in  $A$ is a pair $(L_p, R_p)$
    of (possibly empty) subsets of $A$ such that $L_p<R_p$ and
    $L_p\cup R_p=A$.
  \item If $D$ is a subset of $A$, by \emph{the cut right} (resp.~\emph{left}) \emph{of $D$} we mean the unique cut in $A$ such that $D$ is cofinal in $L_p$ (resp.~coinitial in $R_p$).
  \item If $h\from A\to A$ is an increasing (resp.~decreasing) bijection, the \emph{pushforward} $h_*p$ is the unique cut in $A$ with $L_{h_*p}=\set{h(a): a\in L_p}$ (resp.~$L_{h_*p}=\set{h(a): a\in R_p}$).
  \item The \emph{cut of} $b\in B\setminus A$ in $A$ is the cut $p$ with $L_p\coloneqq\set{a\in A : a<b}$ and $R_p\coloneqq\set{a\in A : a>b}$. We also say that $b$ \emph{realises} $p$, or that $b$ \emph{is in} $p$.
  \item If $p$ is a cut in $A$, the set of its realisations in $B$ is denoted by $p(B)$.
 \end{enumerate}
\end{defin}
We think of cuts in an $\mathbb R\mathsf{-OVSA}$ as \emph{$1$-types} in the reduct to $L_{\mathbb R}$; in practice, this means that, if $I=[a,b]$ is an interval in $A$, we write $p(x)\proves x\in I$ to mean that $a\in L_p$ and $b\in R_p$, and that 
if $B\supseteq A$ and $b\in B$ we write $b\models p$ to mean that $b\in p(B)$.
\begin{defin}
  Let $f\in \mathbb R[\sigma]$ be absolutely monotone  and $d\in A$. Suppose that $f(x)=d$ has no solution in $A$.  The \emph{cut of a zero of $f(x)=d$} is  the cut $p$ in $A$ defined as follows.
  \begin{enumerate}
  \item If $f(x)$ is absolutely increasing, we set $L_p\coloneqq\set{x\in A: f(x)<d}$ and $R_p\coloneqq\set{x\in A: f(x)>d}$.
  \item If $f(x)$ is absolutely decreasing, we set $L_p\coloneqq\set{x\in A: f(x)>d}$ and $R_p\coloneqq\set{x\in A: f(x)<d}$.
  \end{enumerate}
\end{defin}

\begin{eg}\label{eg:norootofabsmon}
Let $B=\mathbb{R}((t))$ be the $\mathbb R\mathsf{-OVS}$ of real Laurent series 
and let $A$ be its subspace $\mathbb R[t,t\inverse]$, consisting of the elements with finite support. Equip $B$ with the automorphism $\sigma$ given by $\sigma(\sum_{i} a_i t^i)= \sum_{i} a_i t^{i+1}$, and $A$ with its restriction, still denoted by $\sigma$. Let $c\coloneqq t^0$.

 The equation $\sigma(x)+x=c$ is solved by  $b\coloneqq \sum_{i\geq 0} (t^{2i}-t^{2i+1})\in B\setminus A$. This solution is unique because the $\sigma$-polynomial $\sigma(x)+x$ is absolutely increasing, hence $\sigma(x)+x=c$ has no solution in $A$. The cut of a zero of $\sigma(x)+x=c$ in $A$ is the one of $b$, namely $(L_p,R_p)=(\{a\in A: x<b\}, \{x\in A: x>b\})$.

\end{eg}

\begin{pr}\label{pr:naacsg}
  Let $f\in \mathbb R[\sigma]$ be absolutely monotone and $d\in A$ be such that $f(x)=d$ has no solution in $A$. Then, the cut in $A$ of a zero of $f(x)=d$ cannot be the cut right or left of a convex subgroup of $A$.
\end{pr}
\begin{proof}
Let $H$ be a convex subgroup of $A$. By replacing the order  $<$ on $A$ with the opposite order $>$ if necessary, it is enough to consider the cut $p$ right of $H$. Moreover, by replacing $f$ with $-f$ and $d$ with $-d$ if necessary, we may assume $f$ is absolutely increasing. Composing $f$ with $\sigma^k$ for a suitable integer $k$ we may assume $f$ is of the form $\sum_{i=0}^{\deg f} \lambda_i\sigma^i(x)$ with $\lambda_i\in \mathbb{R}$ and $\lambda_0\neq 0$ (note $\lambda_{\deg f}>0$ as $f$ is absolutely increasing).
Observe that $\sigma(H)$ is a convex subgroup of $A$. Let $a=\sigma^{-\deg f}(d)\in A$.  Because $f(x)=d$ has no solution in $A$, we must have $d\ne 0$, hence $a\ne 0$ as well.

\textbf{Case 1.} We have $\sigma(a)\asymp a$ or $a\ll \sigma(a)$.  

Then, by \Cref{pr:famighnprp}, point~\ref{point:absmonmoreoverb}, $f(a)\asymp \sigma^{\deg f}(a)=d$, so there is  $k\in \omega\setminus\set 0$ such that $f(\frac 1k a)=\frac 1k f(a)<d<kf(a)=f(ka)$, i.e.~$\frac 1ka\in L_p$ and $ka\in R_p$, so $\frac 1ka\in H$ and $ka>H$, a contradiction.

\textbf{Case 2.} We have $\sigma(a)\ll a$.

Then $f(a)\asymp a$, so $f(\sigma^{\deg f}(a))=\sigma^{\deg f}(f(a))\asymp \sigma^{\deg f}(a)=d$, so again there is  $k\in \omega\setminus \set 0$ with $f(\frac 1k \sigma^{\deg f}(a))=\frac 1k f(\sigma^{\deg f}(a))<d<kf(\sigma^{\deg f}(a))=f(k\sigma^{\deg f}(a))$, a contradiction as above.
\end{proof}
\begin{fact}\label{fact:lifttoi-completion}
  If $A\models \mathbb R\mathsf{-OVSA}$ is non-trivial, then there is a maximally complete immediate extension $B\models \mathbb R\mathsf{-OVSA}$ of $A$.
\end{fact}
\begin{proof}
It suffices to take as $B$ the Hahn product of the skeleton of $A$, together with the lift of $\sigma$ provided by \Cref{fact:HET}.
\end{proof}

\begin{fact}\label{fact:cutinicomplete}
  If $A\models \mathbb R\mathsf{-OVSA}$ is maximally complete, then every cut in $A$ is a translate of a cut left or right of a convex subgroup.
\end{fact}
\begin{proof}
  Recall that an extension $B\subseteq C$ of ordered $\mathbb Q$-vector spaces is an \emph{i-extension} iff the map $H\mapsto H\cap B$ is a bijection between the convex subgroups of $C$ and $B$, and that $B$ is \emph{i-complete} iff it has no proper i-extensions. If $A$ is an ordered $\mathbb R$-vector space, it is easy to see that it is i-complete if and only if it is maximally complete with respect to the Archimedean valuation, if and only if, by \Cref{fact:HET}, it is isomorphic to the Hahn group $\mathbb R((I))$, where $I$ is the spine of $A$. It now suffices to regard $A$ as an i-complete ordered $\mathbb Q$-vector space and invoke~\cite[Corollary~13.11]{hhm} (see also~\cite[Proposition~4.8]{dominomin}).
\end{proof}

\begin{co}\label{co:ivp_mon}
If $A\models \mathbb R\mathsf{-OVSA}$ is maximally complete, or existentially closed, then for all absolutely monotone $f\in \mathbb R[\sigma]$ the induced $f(x)\from A\to A$ is surjective.
\end{co}
\begin{proof}
  It is enough to prove the conclusion in the maximally complete case, since the existentially closed case will then follow from \Cref{fact:lifttoi-completion}.
So, let $A$ be maximally complete, and assume towards a contradiction that $f(x)-d$ has no solution in $A$.   
Let $p(x)$ be the cut of a zero.  By \Cref{fact:cutinicomplete} we find $\lambda\in \set{1, -1}$ and $a\in A$ such that, if $h(x)=\lambda  x-a$, then the pushforward $h_*p$ is the cut right of a convex subgroup.  If $\lambda=1$, then
\begin{equation}
  b\models f(x)= d\iff  b-a\models f(y+a) = d\iff b-a\models f(y)= d-f(a)\label{eq:tcaos}  
\end{equation}
Hence, if $f(x)=d$ does not have a solution in $A$, neither does $f(y)=d-f(a)$.
Moreover,  \eqref{eq:tcaos} also holds after replacing every $=$ by $\ge$, or every $=$ by $\le$, which implies that $h_*p$ is the cut of a zero of $f(y)=d-f(a)$, contradicting \Cref{pr:naacsg}. The case where $\lambda=-1$ is similar.
\end{proof}

\section{The Intermediate Value Property}\label{sec:ivp}
\subsection{The IVP for $\sigma$-polynomials\ldots}
In this subsection we prove that, if $M\models \mathbb R\mathsf{-OVSA}$ is existentially closed, then every $\sigma$-polynomial, when viewed as a function $M\to M$, has the Intermediate Value Property, that is, it maps convex sets to convex sets.

Clearly, since we are working on an oag, saying that every $\sigma$-polynomial has the IVP is the same as saying that whenever a $\sigma$-polynomial over $M$ changes sign on an interval, then it has a zero in the same interval. Without existential closedness, this can indeed fail.
\begin{eg}\label{eg:noIVP}
Let $J$ be the concatenation of two copies $I_0, I_1$ of $\mathbb Z$, and let $A$ be the Hahn group $\mathbb R((J))$. Consider the automorphism $\sigma$ induced by backward shift on $I_0$ and forward shift on $I_1$. If $a,b\in A$ are positive and such that $\supp(a)\subseteq I_1$ and $\supp(b)\cap I_0\ne \emptyset$, then $\sigma(a)< a<b< \sigma(b)$. It follows that the $\sigma$-polynomial $h(x)=\sigma(x)-x$ changes sign on $[a,b]$, hence the IVP would grant a zero of $h$, that is, a fixed point of $\sigma$, in $[a,b]$. Nevertheless, since elements of $\mathbb R((J))$ have well-ordered support, it follows easily that $h$ has no non-zero fixed points in $A$ because, by construction, if $c\ne 0$ then $\min(\supp(c))\ne \min(\supp(\sigma(c)))$.
\end{eg}

\begin{rem}\label{rem:ivpcomp}
The class of functions $A\to A$ with the IVP is closed under composition.
\end{rem}

\begin{defin}\label{defin:csaac}
Let $p$ be a cut in $A\models \mathbb R\mathsf{-OVSA}$ and $h(x)$ a $\sigma$-polynomial over $A$. We say that $h$ \emph{changes sign at $p$} if
\begin{itemize}
\item $\set{a\in A: h(a)<0}\cap L_p$ is cofinal in $L_p$ and $\set{a\in A: h(a)>0}\cap R_p$ is coinitial in $R_p$; or
  \item $\set{a\in A: h(a)>0}\cap L_p$ is cofinal in $L_p$ and $\set{a\in A: h(a)<0}\cap R_p$ is coinitial in $R_p$.
\end{itemize}
\end{defin}
\begin{rem}\label{rem:newton}
Assume that $a<b$ and that $h(a)$ and $h(b)$ have different sign. A routine transfinite argument shows that either there is a zero of $h$ in $[a,b]$, or there is a cut $p$ with $a\in L_p$ and $b\in R_p$ such that $h$ changes sign at $p$.
\end{rem}
\begin{thm}\label{thm:ivp}
If $h\in \mathbb R[\sigma]$, and $M$ is existentially closed, then $h(x)\from M\to M$ has the Intermediate Value Property.
\end{thm}
\begin{proof}
We begin by proving that, if $A\models \mathbb R\mathsf{-OVSA}$ and $h$ is a $\sigma$-polynomial of degree $1$ over $A$ then, whenever $p$ is a cut in $A$ where $h$ changes sign, in some extension of $A$ there is a zero of $h$ lying in $p$.
  
 Let $B_0$ and $B_1$ be $\mathbb R((\mathbb Z))$ with the automorphism induced by forward and backward shift respectively. Up to replacing $A$ by $B_i\times_\mathrm{lex} A$ for some $i\in \set{0,1}$ we may assume that $p$ is not at  $\pm\infty$, that is,  $L_p\ne \emptyset\ne R_p$. Since the cuts at $\pm\infty$ are fixed by (the pushforward along) every automorphism,  $\sigma_*p$ is also not at $\pm \infty$. Moreover, because $h$ is continuous\footnote{In the order topology.}, it follows that $L_p$ has no maximum and $R_p$ has no minimum.
  
Up to rescaling, we may assume $h(x)=\sigma(x)-g(x)$ with $g(x)$ affine, that is, of the form $\lambda  x+d$, with $d\in A$. 
\begin{claim}\label{claim:ivpproof}
  It is enough to show that, for every closed interval $I_0=[\ell_0, r_0]$ such that $p(x)\proves x\in I_0$, and every closed interval $I_1=[\ell_1, r_1]$ such that $\sigma_* p(x)\proves x\in I_1$, there is $a_0\in I_0$ such that  $g(a_0)\in I_1$. 
\end{claim}
\begin{claimproof}
  By a standard model-theoretic fact,  the restriction of  $A$ to the language $L_{\mathbb R}$ embeds in a \emph{monster model} $\monster\models \mathbb R\mathsf{-OVS}$.  We will come back to this concept in \Cref{conv:monster-model} but, for the time being, all the non-model-theorist needs to know is that every cut in $A$ is realised by some point of $\monster$, and that every partial \emph{elementary map} with domain of size not larger than that of $A$ extends to an automorphism of $\monster$. We will not define here what an elementary map is in general; for models of $\mathbb R\mathsf{-OVS}$, due to a phenomenon known as \emph{quantifier elimination}, a map between subsets $A_0, B_0$ of $\monster$ is elementary if and only if it extends to an isomorphism between the ordered vector subspaces they generate.
  
 By compactness, the assumptions of the claim guarantee that $p(x)\cup \sigma_*p(g(x))$ is consistent.\footnote{Of course, since $p$ is a complete $L_{\mathbb R}(A)$-type, this is equivalent to $p(x)\proves \sigma_*p(g(x))$.} Pick an arbitrary $b_0\models p(x)\cup \sigma_*p(g(x))$ in $\monster$ and set $b_1\coloneqq g(b_0)$. It follows that sending each $a\in A$ to $\sigma(a)$ and $b_0$ to $b_1$ is an elementary map, that we can extend to an automorphism of $\monster$.
\end{claimproof}
Therefore, it suffices to prove that the properties in \Cref{claim:ivpproof} are satisfied. Up to restricting $I_0$ and $I_1$, we may assume  that $\ell_{1}=\sigma(\ell_0)$ and $r_{1}=\sigma(r_0)$. In the case where  $\set{a\in A: h(a)<0}\cap L_p$ is cofinal in $L_p$ and $\set{a\in A: h(a)>0}\cap R_p$ is coinitial in $R_p$, we may furthermore assume  $h(\ell_0)<0<h(r_0)$.
 So
\begin{equation}\label{eq:lnrn}
  \begin{split}
\ell_1= \sigma(\ell_0)< g(\ell_0)\\ r_1= \sigma(r_0)> g(r_0)
\end{split}
\end{equation}

If at least one of $g(\ell_0)$ and $g(r_0)$ is in $I_1=[\ell_1, r_1]$, we conclude by setting $a_0\coloneqq\ell_0$ or $a_0\coloneqq r_0$. Otherwise, by~\eqref{eq:lnrn} we must have $g(r_0)< \ell_1< r_1<g(\ell_0)$. But then the image of the closed interval $I_0$ under the affine function $g$, which is a closed interval, contains $\ell_1, r_1$, and the conclusion follows by choosing some $a_0\in I_0$  such that $g(a_0)\in [\ell_1, r_1]$.

In the other case we may assume that $h(\ell_0)>0>h(r_0)$, which means that $\ell_1= \sigma(\ell_0)> g(\ell_0)$ and  $r_1= \sigma(r_0)< g(r_0)$. Again, this implies that $[\ell_1, r_1]\subseteq g(I_0)$, and we conclude as above.

This shows that every $\sigma$-polynomial of degree $1$ over an existentially closed model has the Intermediate Value Property. In the general case, let $h\in \mathbb R[\sigma]$ be of arbitrary degree, and let $\tilde h(y)$ be the polynomial associated to $h$. If $\tilde h\in \mathbb R$ then $h$ is a multiple of the identity and the conclusion is trivially true; otherwise, we may factorise $\tilde h$ as  $\tilde h(y)=\prod_{i\leq n} \tilde f_i(y)$, where the $\tilde f_i$ are irreducible real polynomials; in particular,  $\deg(\tilde f_i)\in \set{1,2}$. Let $f_i$ be the real $\sigma$-polynomial with associated polynomial $\tilde f_i$. Then $h$ is equal to the composition $f_0f_1\ldots f_n$, and by \Cref{rem:ivpcomp} it is enough to prove that each $f_i$ has the IVP. But we have already proven this for those $i$ such that $\deg(\tilde f_i)=1$, and if $\tilde f_i$ is irreducible of degree $2$, then $f_i$ is absolutely monotone by \Cref{pr:famighnprp}, hence  surjective by  \Cref{co:ivp_mon}, and every monotone, surjective function has the IVP.
\end{proof}

\begin{co}\label{co:ivp_pol}
  If $M$ is existentially closed, every  $h\in \mathbb R[\sigma]\setminus\set 0$ is surjective.
\end{co}
\begin{proof}
By  \Cref{lemma:ubddimpec} and \Cref{thm:ivp}.
\end{proof}
\subsection{\ldots and minima thereof}
In this subsection we  extend \Cref{thm:ivp} by proving (\Cref{thm:minima}) that, if $M\models \mathbb R\mathsf{-OVSA}$ is existentially closed, and  $h$ is the minimum of finitely many $\sigma$-polynomials over $M$, then $h$ has the IVP on $M$. 
\begin{notation}Let $h=\min(f_0,\ldots,f_{n-1})$ for some $\sigma$-polynomials $f_0,\ldots,f_{n-1}$ over $A\models \mathbb R\mathsf{-OVSA}$.
  \begin{enumerate}
  \item We extend the terminology introduced in \Cref{defin:csaac} to
    such $h$.
  \item Extending \Cref{defin:tilde}, we write $\tilde{h}\coloneqq\min(\tilde f_0 ,\ldots,\tilde f_{n-1})$.
  \item We denote the $\sigma$-orbit of $a\in A$ by $\orb(a)\coloneqq\set{\sigma^k(a): k\in \mathbb Z}$.
  \item   We say that a convex subgroup $H\le A\models \mathbb R\mathsf{-OVSA}$ is \emph{invariant} iff it is stabilised setwise by $\sigma$, that is, iff $\sigma(H)=H$.
\end{enumerate}
\end{notation}
\begin{lemma}\label{lemma:add_c}
Let $B\models \mathbb{R}\mathsf{-OVSA}$ and let $p$ be the cut right of an invariant convex subgroup $H\leq B$. For every $\rho>0$ there is an extension $B\subseteq C\models \mathbb R\mathsf{-OVSA}$ generated as an $\mathbb{R}$-vector space by $B$ and some  $c\in p(C)$ such that  $\sigma(c)=\rho   c$.
\end{lemma}
\begin{proof}
Let $C$ be generated by $B$ and $c$ as an $\mathbb{R}$-vector space with $c\notin B$, ordered as follows. For every $\lambda\in \mathbb{R}$ and $b\in B$, set $\lambda c+b>0$ if and only if $(\lambda>0\wedge -\frac 1 \lambda b\in L_p)\vee (\lambda<0\wedge -\frac 1\lambda b\in R_p)\vee (\lambda=0\wedge  b>0) $. It is routine to check, either by using monster models or directly by hand, that this gives an ordered $\mathbb{R}$-vector space structure on $C$, and that $\lambda c+b\mapsto \rho \lambda c+\sigma(b)$ is its automorphism.
\end{proof}
\begin{defin}
Let $p$ be a cut in $A\models \mathbb R\mathsf{-OVSA}$ and $f$ a $\sigma$-polynomial over $A$. We say that $f$ \emph{potentially changes sign} at $p$ iff for every $a\in L_p$ and $b\in R_p$ there is an extension $A\subseteq B\models \mathbb R\mathsf{-OVSA}$ such that $f$ assumes different signs in $B$ between $a$ and $b$.
\end{defin}
\begin{lemma}\label{lemma:pcsec}
Let $f$ be a $\sigma$-polynomial over $A$, and $p$ a cut in the latter.  The following are equivalent.
  \begin{enumerate}
  \item\label{point:potcs} The $\sigma$-polynomial $f$ potentially changes sign at $p$.
  \item\label{point:extds} In some extension $B$ of $A$, there are points of $p(B)$ where $f$ assumes different signs.
  \item\label{point:eczipm} In some existentially closed extension $M$ of $A$, there is a zero of $f$ in $p(M)$.
  \end{enumerate}
\end{lemma}
\begin{proof}
Compactness gives immediately $\ref{point:potcs}\allora\ref{point:extds}$, while $\ref{point:extds}\allora\ref{point:potcs}$ is trivial and $\ref{point:extds}\allora\ref{point:eczipm}$ follows from  \Cref{thm:ivp}. To prove $\ref{point:eczipm}\allora\ref{point:extds}$, if $m\in p(M)$ is such that $f(m)=0$, by passing to a suitable lexicographical product we find $\epsilon$ such that $f(m+\epsilon)>0$ and $f(m-\epsilon)<0$.
\end{proof}

\begin{lemma}\label{lemma:in_H}
Let $H$ be an invariant convex subgroup of $A\models \mathbb R\mathsf{-OVSA}$. For all $d\in A$ and $f\in \mathbb R[\sigma]$, if $f(x)+d$ potentially changes sign at the cut right of $H$, then $d\in H$.
\end{lemma}
\begin{proof}
Suppose not,  let $p$ be the cut right of $H$,  let $M$ be given by \Cref{lemma:pcsec}, and set $k\coloneqq\deg(f)$. 
If $\sigma(\abs{d})\gg \abs{d}$ then, for every  $c\in M$ such that  $0<c< \sigma^{-k-1}(\abs{d})$, we have   $\abs{d}>\abs{f(c)}$.   Therefore, for such $c$ we have  $\sign(f(c)+d)=\sign(d)$. Since  $\sigma^{-k-1}(\abs{d})>H$ we conclude that  $f(x)+d$ does not potentially change sign at $p$. Similarly, if $\sigma(\abs{d})\ll d$ and $0<c<\sigma( \abs{d})$, then $\abs{d}>\abs{f(c)}$, so again $f(x)+d$ does not potentially change sign at $p$. 

In the remaining case, $\sigma(d)\sim \rho   d$ for some positive $\rho \in \mathbb{R}$.
If $c\in M$ and $0<c\ll \abs{d}$, then $\sigma^n(c)\ll \abs{d}$ for every $n\in \omega$, so $\abs{d}>\abs{f(c)}$. Also, if  $c>0$ and  $d\asymp c<\frac 1 {1+\abs{\tilde f(\rho )}}\abs{d}$,
 we have that $c\sim \lambda\abs{d}$ for 
some $0<\lambda\leq  \frac 1 {1+\abs{\tilde f(\rho )}}$. Therefore, if $\tilde f(\rho )\neq 0$ then $ \abs{f(c)}\sim \lambda \abs{\tilde f(\rho )}\abs{d}\leq \frac {\abs{\tilde f(\rho )}}{\abs{\tilde f(\rho )}+1}\abs{d}$, 
and if $\tilde f(\rho )= 0$ then $f(c)\ll \abs{d}$. In either case  $\abs{d}>\abs{f(c)}$,  thus $f(c)+d$ has the same sign as $d$ for $0<c<\frac 1 {1+\abs{\tilde f(\rho )}}\abs{d}$, hence $f(x)+d$ does not potentially change sign at $p$.
\end{proof}

\begin{lemma}\label{lemma:pos_root}
Let $A\models \mathbb R\mathsf{-OVSA}$ and $h=\min(f_0,\ldots,f_{n-1})$ for some non-absolutely monotone $\sigma$-polynomials $f_0,\ldots,f_{n-1}\in \mathbb R[\sigma]$. If $h$ assumes both positive and negative values on $\{a\in A: a>0\}$, then there is some positive real number $\rho$ with $\tilde{h}(\rho)=0$.
\end{lemma}
\begin{proof}
Note no $f_i$ is the null $\sigma$-polynomial as $h$ assumes positive values.
For all $x_0,\ldots,x_{n-1}\in A$ and every $k\in \mathbb{Z}$  the sign of  $\min(x_0,\ldots, x_{n-1})$ is the same as the sign of $\min(x_0,\ldots,x_{i-1},\sigma^k(x_i),x_{i+1},\ldots,x_{n-1})$, so by composing each $f_i$ with $\sigma^k$ for a suitable $k$ we may assume that for every $i<n$ we have $\tilde f_i(0)\neq 0$. 

Suppose there is no positive $\rho$ with $\tilde h(\rho) =0$. Note that we cannot have $\tilde h(\rho) >0$ for every $\rho>0$, as then $\tilde f_i(\rho)>0$ for every $i<n$ and $\rho>0$, hence $f_i$ would be absolutely monotone by \Cref{pr:famighnprp}. Thus, as $\tilde h$ is a continuous real function, we get that $\tilde h(\rho)<0$ for every $\rho>0$. Hence, for some $i<n$, we must have $\lim_{x\to \infty} \tilde f_i(x)=-\infty$, as otherwise we would have $\lim_{x\to \infty} \tilde f_i(x)=\infty$ for every $i$ (note $\tilde f_i$ cannot be constant as $f_i$ is not absolutely monotone), so $\lim_{x\to \infty}\tilde h(x)=\infty$, a contradiction. Similarly, for some $i'<n$, we must have $\tilde f_{i'}(0)=\lim_{x\to 0}\tilde f_{i'}(x)<0$ (recall that we arranged $\tilde f_i(0)\neq 0$). Note $\lim_{x\to \infty}\tilde f_i(x)=-\infty$ means that the leading coefficient of $f_i$ is negative,  so if $\sigma(a)\gg a$ then $h(a)\leq f_i(a)<0$. On the other hand, $\tilde f_{i'}(0)<0$ means that the coefficient of $x$ in $f_{i'}$ is negative, so for  $\sigma(a)\ll a$ we have $h(a)\leq f_{i'}(a)<0$. Also, if  $a>0$ and $\sigma(a)\sim \rho  a$ for some real number $\rho> 0$ then, as $\tilde h(\rho)<0$, we have $h(a)\sim \tilde h(\rho)  a<0$. Thus $h(a)<0$ for every positive $a$, contradicting the assumption.
\end{proof}
\begin{thm}\label{thm:minima}
Let $h=\min(f_0,\ldots,f_{n-1})$ for some $\sigma$-polynomials $f_0,\ldots,f_{n-1}$ over $A\models \mathbb R\mathsf{-OVSA}$. Suppose $a,b\in A$ with $a<b$ are such that $h$ changes sign between $a$ and $b$. Then, in some extension of $A$, there is a zero of $h$ between $a$ and $b$. Equivalently, if $M\models \mathbb R\mathsf{-OVSA}$ is existentially closed, then every $h$ as above has the IVP on $M$.
\end{thm}
\begin{proof}
 By induction on $(n, \max_{i< n}\deg f_i)\in (\omega\setminus\set0)\times_\mathrm{lex}\omega$. If $n=1$, this is \Cref{thm:ivp}. Consider $h=\min(f_0,\ldots,f_{n-1})$ and $a<b$  such that $h$ changes sign between $a$ and $b$. Suppose for a contradiction $h,a,b$ do not satisfy the conclusion.

\begin{claim}\label{cl:not_monotone}
No $f_i$ is absolutely monotone.
\end{claim}
\begin{claimproof}
Suppose for example $f_0$ is absolutely monotone.  Without loss of generality, assume that $A$ is existentially closed. Then, by point~\ref{point:absmononesol} of \Cref{rem:amsscoi} and \Cref{co:ivp_mon}, $f_0$ has exactly one root $m$. If $m\notin (a,b)$, then $f_0> 0$ on $(a,b)$ (as otherwise $h\leq f_0<0$ on $(a,b)$, so $h$ cannot change sign between $a$ and $b$),  so  a root of $\min(f_1,\ldots,f_{n-1})$ on $(a,b)$, which exists by the inductive hypothesis, will also be a root of $h$. So assume $m\in (a,b)$. Then $h$ changes sign between $a$ and $m$, or between $m$ and $b$, and in each case we can conclude as above, as $m\notin (a,m)$ and $m\notin (m,b)$.
\end{claimproof}

If $0\in (a,b)$, then $h$ changes sign either between $a$ and $0$ or between $0$ and $b$, so by replacing $(a,b)$ with $(a,0)$ or $(0,b)$, we may assume $0\notin (a,b)$. 

By \Cref{fact:lifttoi-completion} we may assume that $A$ is maximally complete.
As in \Cref{rem:newton}, we can find a cut $p$ between $a$ and $b$ (i.e.~$a\in L_p$ and $b\in R_p$) at which $h$ changes sign. By \Cref{fact:cutinicomplete}, $p$ is the cut either left or right of $v+H$ for some convex subgroup $H\leq A$ and some $v\in A$. By replacing each $f_i(x)$ with $f_i(x+v)$, we may assume $v=0$. Furthermore, by reversing the order on $A$ and replacing each $f_i$ with $-f_i$ if necessary, we may assume $p$ is the cut right of $H$, as $\min(f_0,\ldots,f_{n-1})$ calculated with respect to the reversed order equals $\max(f_0,\ldots,f_{n-1})=-\min(-f_0,\ldots,-f_{n-1})$ in the original order. 

\begin{claim}\label{claim:sign}
Every $f_i$ potentially changes sign at $p$.
\end{claim}
\begin{claimproof}
Suppose not, as witnessed by, say, $f_0$. If $f_0$ is negative on an interval $(a',b')$ around $p$, that is, with $a'\in L_p$ and $b'\in R_p$, then $h$ does not change sign at $p$. Therefore, there is  an interval $(a',b')$ around $p$ where $f_0$ is non-negative; up to shortening this interval, we may assume that $f_0$ stays non-negative on $(a',b')$ in every extension of $A$.
By the inductive hypothesis, in some existentially closed extension of $A$, there is a zero of $\min(f_1,\ldots,f_{n-1})$ on the interval $(a,b)\cap (a',b')$, and by choice of $(a',b')$ this is also a zero of $h$, contradicting the choice of $h,a,b$.
\end{claimproof}
For each $i<n$, write $f_i=f_i^*+d_i$ for a $\sigma$-polynomial $f^*_i$ and $d_i\in A$, as in \Cref{defin:tilde}. As $h$ changes sign at $p$, by \Cref{claim:sign} and \Cref{lemma:in_H} we have  $d_0,\ldots,d_{n-1}\in H$.
\begin{claim}
The convex subgroup $H$ is invariant, i.e.~$\sigma(H)=H$.
\end{claim}
\begin{claimproof}
By \Cref{claim:sign}, $f_0$ potentially changes sign at $p$, and hence by \Cref{lemma:pcsec} there are  $M\supseteq A$ and $c\in p(M)$ with $f_0(c)=0$. If $\sigma(H)$ is a convex subgroup of $A$ distinct from $H$, we must have either $\sigma(c)\gg c$ or $\sigma(c)\ll c$. Thus $-d_0=f_0^*(c)\asymp \sigma^k(c)$ for some $k\in\mathbb{Z}$, so $c\asymp \sigma^{-k}(-d_0)\in A$, which is not possible as $c$ realises a cut right of a convex subgroup.
\end{claimproof}

\begin{claim}\label{claim:rho}
There is some positive real number $\rho$ with $\tilde{h}(\rho)=0$.
\end{claim}
\begin{claimproof}
Suppose there is no such $\rho$. By \Cref{cl:not_monotone} and \Cref{lemma:pos_root}, we get a contradiction as soon as we check that $h^*\coloneqq \min(f_0^*,\ldots,f_{n-1}^*)$ assumes both positive and negative values on the positive elements of some model of $ \mathbb R\mathsf{-OVSA}$.

We will deal with the case where $h^*(b)\geq 0$, the argument in the case where $h^*(b)\leq 0$ being entirely analogous. 
Let $d'\coloneqq \max(\abs{d_0},\ldots,\abs{d_{n-1}})$.  As  $d_0,\ldots,d_{n-1}\in H$, we have that $b\gg \orb(d')$. If $\sigma(b)\sim \rho  b$ for some $\rho>0$, then, as $\tilde{h}(\rho)\neq 0$ by assumption, we have $\abs{h^*(b)} \sim \abs{\tilde{h}(\rho)  b}\asymp \abs{b}$ so $h^*(b)> 0$.  Similarly, when $\sigma(b)\gg b$ or $\sigma(b)\ll b$, we get that $h(b)\asymp \sigma^k(b)$ for some $k$, hence again $h^*(b)\geq 0$ implies that  $h^*(b)>0$.

Since $h$ changes sign at $p$ and $d_0,\ldots,d_{n-1}\in H$, we can  find, for every $k\in \mathbb{Z}$ and $\ell\in \omega$, some $a'\in A$ with $a'>\ell\sigma^k(d')$ and $h(a')$ negative. Hence, by compactness, there is some extension  $A\subseteq B \models \mathbb {R}\mathsf{-OVSA}$ and a positive $b'\in B$ with $b'\gg\orb(d')$ and $h(b')$ negative. Then, as in the previous paragraph, the assumption that $\tilde h$ has no positive zeroes implies that $h^*(b')$ is negative as well. Therefore, $h^*$ assumes both positive and negative values on the positive elements of $B$.
\end{claimproof}

Let $\rho$ be given by \Cref{claim:rho}. If $\tilde f_i (\rho)=0$ for every $i< n$, then $f_i(x)=g_i(\sigma(x)-\rho x)$ for some $\sigma$-polynomial $g_i$ over $A$ with $\deg(g_i)=\deg(f_i)-1$ and $h=h_0\circ h_1$ where $h_0=\min(g_0,\ldots,g_{n-1})$ and $h_1(x)=\sigma(x)-\rho x$. Now if $M\supseteq A$ is existentially closed, then, as $\max_i\deg(g_i)<\max_i \deg(f_i)$, we get by the inductive hypothesis that $h_0$ has the IVP on $M$, and so does $h_1$ by \Cref{thm:ivp}. Thus $h$ has the IVP on $M$ by \Cref{rem:ivpcomp}, so there is some $m\in M$ with $a<m<b$ and $h(m)=0$, contradicting the choice of $h,a,b$.

Hence, we may assume that for some $k<n$ we have $\tilde f_0(\rho)=\ldots=\tilde f_{k-1}(\rho)=0$ and $\tilde f_k(\rho),\ldots,\tilde f_{n-1}(\rho)> 0$. 
Put $h_k\coloneqq \min(f_0,\ldots,f_{k-1})$.
\begin{claim}
  There are $B\supseteq A$ and $c_0,c_1\in p(B)$  with  $h_k(c_0)<0<h_k(c_1)$. 
\end{claim}
\begin{claimproof}
  By  \Cref{claim:sign} and \Cref{lemma:pcsec} there is $B'\supseteq A$ containing some $c_0\in p(B')$ such that $f_0(c_0)<0$, hence $h_k(c_0)<0$. Suppose there are  $a'\in L_p$ and $b'\in R_p$ (in particular, $a',b'\in A$) such that, on every point of $B'$ in the interval $(a',b')$, the function $h_k$ is non-positive. Then, a fortiori, $h_k$ is non-positive on $(a',b')\cap A$, hence so is $h$, contradicting that the latter changes sign at $p$. Therefore, by compactness, in some extension $B$ of $B'$ there is $c_1\in p(B)$ such that $h_k(c_1)>0$.
\end{claimproof}
 Now let $q$ be the cut in $B$ left of  $\{z\in A: z>H\}$. Observe that $q$ is a cut right of an invariant convex subgroup of $B$. Hence, by \Cref{lemma:add_c}, there is a structure $C$, generated as an $\mathbb{R}$-vector space by $B\cup \{c\}$, with $c$ realising $q$ and $\sigma(c)=\rho  c$. Note $a\ll c_0,c_1\ll c\ll b$, and $d_0,\ldots,d_{n-1}\ll c$ because, as we already pointed out, $d_0,\ldots,d_{n-1}\in H$. 
Now $h_k(c+c_1)>0$, as for every $i<k$ we have that $f_i(c+c_1)=f^*_i(c+c_1)+d_i=f^*_i(c)+f^*_i(c_1)+d_i=f^*_i(c)+f_i(c_1)>f^*_i(c)=\tilde f_i(\rho)c =0$. As $h_k(c_0)<0$, there is some $i<k$ such that $f_i(c_0)<0$, so, using again that $f^*(c)=0$, we get $h_k(c+c_0)\leq f_i(c+c_0)=f_i(c_0)<0$.  Hence, by the inductive assumption, there are some extension $D\supseteq C$ and $c_2\in D$ between $c_0$ and $c_1$ with $h_k(c_2)=0$. Note that $\abs{c_2}\ll c$, hence for each $j>k$ we have $f_j(c+c_2)=\tilde f_j(\rho)c+f_j(c_2)>0$, as $\sigma^k(c_2)\ll c$ for every $k\in \mathbb{Z}$, $d_j\ll c$, and $\tilde f_j(\rho)>0$.
Thus $h(c+c_2)=\min(h_k(c+c_2),f_k(c+c_2),\ldots, f_{n-1}(c+c_2))=\min(0,f_k(c+c_2),\ldots, f_{n-1}(c+c_2))=0$. As $ a\ll c\ll b$ and $\abs{c_2}\ll c$, we get $a<c+c_2< b$, as required.
\end{proof}

\begin{co}\label{co:ivp_min}
On every existentially closed $M\models \mathbb Q\mathsf{-OVSA}$, minima of finitely many $\sigma$-polynomials over $M$ with rational coefficients have the IVP.
\end{co}
\begin{proof}
  Use \Cref{thm:redtoR} to extend $M$ to some $B\models \mathbb R\mathsf{-OVSA}$, then extend $B$ to an existentially closed $N\models \mathbb R\mathsf{-OVSA}$. Every $\sigma$-polynomial over $M$ with coefficients in $\mathbb Q$ is in particular one with coefficients in $\mathbb R$, hence, by \Cref{thm:minima}, the minimum of finitely many of them has the IVP on $N$, and the conclusion follows.
\end{proof}

\section{The Amalgamation Property}\label{sec:amalgam}
In this section we prove the main result of the paper, namely that the categories of models of $\mathbb R\mathsf{-OVSA}$, of models of $\mathbb Q\mathsf{-OVSA}$, and of oags with an automorphism, where the arrows are the embeddings in the appropriate language, have the Amalgamation Property. We deal with amalgamating $\sigma$-algebraic points (\Cref{defin:sigmaalgebraic}) in  \Cref{subsec:alg_ap}, and with amalgamating $\sigma$-transcendental points (and hence with the general case) in \Cref{subsec:tr_ap}.
\begin{rem}\label{rem:amsaeo}
If $A\subseteq B,C$ are models of $\mathbb R\mathsf{-OVSA}$, then $B$ and $C$ amalgamate over $A$ (in the category of models of $\mathbb R\mathsf{-OVSA}$) if and only if there is $B\subseteq D \models \mathbb R\mathsf{-OVSA}$ such that $C$ embeds into $D$ over $A$, if and only if there is $C\subseteq D \models \mathbb R\mathsf{-OVSA}$ such that $B$ embeds into $D$ over $A$.
\end{rem}
If $E$ is a subset of a model of $\mathbb R\mathsf{-OVSA}$, we denote by $\seq{E}_{\sigma}$ the $L_{\mathbb R,\sigma}$-structure generated by $E$. Recall that $\sigma\inverse\in L_{\mathbb R, \sigma}$.
\subsection{Amalgamating $\sigma$-algebraic points}\label{subsec:alg_ap}
\begin{defin}\label{defin:sigmaalgebraic}
Let $E$ be a subset of $B\models\mathbb R\mathsf{-OVSA}$.
\begin{enumerate}
\item We say that $b\in B$ is \emph{$\sigma$-algebraic over $E$} iff $b$ is a
  zero of some non-zero $\sigma$-polynomial over $\seq{E}_{\sigma}$. We say that $b$ is
  \emph{$\sigma$-transcendental over $E$} otherwise. 
\item We denote the set  of points of $B$ which are $\sigma$-algebraic over $E$ by
  $\cl_B(E)$.
\item We call $E$ \emph{$\sigma$-algebraically closed} in $B$ iff $\cl_B(E)=E$.
\end{enumerate}

\end{defin}
\begin{rem}\label{rem:pregeometry}
  It is routine to show that $\cl_B$ is a pregeometry, and that $\cl_B(E)\models \mathbb R\mathsf{-OVSA}$. Note that $\cl_B(\emptyset)$ includes the set of fixed points of $\sigma$, hence may grow with $B$.
\end{rem}
\begin{lemma}\label{lemma:monot_AP}
Suppose $A,B,C\models \mathbb R\mathsf{-OVSA}$, $A\subseteq B$, $A\subseteq C$, and $B=\langle A,b\rangle_{\sigma}$ for some $b$ which is a root of an absolutely monotone $\sigma$-polynomial over $A$. Then $B$ and $C$ amalgamate over $A$.
\end{lemma}
\begin{proof}
Suppose $b$ is a solution to $f(x)=a$, with $f$ an absolutely monotone $\sigma$-polynomial and $a\in A$. By replacing $f$ and $a$ with $-f$ and $-a$ if necessary, we may assume that $f$ is absolutely increasing.
By the IVP (\Cref{thm:ivp}), there are $C\subseteq D\models \mathbb R\mathsf{-OVSA}$ and $d\in D$ with $f(d)=a$.
We claim that $B$ embeds into $D$ over $A$ via $b\mapsto d$. For every $\sigma$-polynomial $h$ and every $e\in A$ we have \[h(b)\geq e\iff f(h(b))\geq f(e)\iff  h(f(b))\geq f(e)\iff h(a)\geq f(e) \]
where the first equivalence follows from the fact that $f$ is absolutely increasing, and the second one from the fact that $f$ commutes with $h$.
 In the same way we get that $h(d)\geq e\iff h(a)\geq f(e)$. Since $f(e)\in A$, this shows that $b\mapsto d$ gives an embedding of $B$ into $D$ over $A$, and we conclude by \Cref{rem:amsaeo}.
\end{proof}

\begin{lemma}\label{lemma:sign_change}
Suppose $A\subseteq C\models \mathbb R\mathsf{-OVSA}$ and $p$ is a cut in $A$. Let  $f(x)=\lambda_0x+\lambda_1\sigma(x)-a$ be a $\sigma$-polynomial over $A$ such that $\lambda_0 \lambda_1<0$. If there are  $A\subseteq B\models \mathbb R\mathsf{-OVSA}$ and $b\in p(B)$ with $f(b)=0$, then there are $C\subseteq D\models \mathbb R\mathsf{-OVSA}$ and $d\in p(D)$ with $f(d)=0$.
\end{lemma}
\begin{proof}
  By scaling $f$ by a real number, we may assume $\lambda_0=1$. By replacing $\sigma$ with an automorphism $\sigma'(x)\coloneqq \rho \sigma(x)$ for some positive real number $\rho$, we may additionally assume that $\lambda_1=-1$, so $f(x)=x-\sigma(x)-a$.
  
\textbf{Case 1.} For every $a_0,a_1\in A$ with $a_0\in L_p$ and $a_1\in R_p$ we have $\{x\in A:f(x)\geq 0\}\cap [a_0,a_1]\neq \emptyset$.

By compactness it is enough to find, for every $a_0,a_1$ as above, an extension $C\subseteq C'\models \mathbb R\mathsf{-OVSA}$ and $c'\in C'$ with $a_0<c'<a_1$ and $f(c')=0$.
So fix such  $a_0,a_1$. If there is $c\in p(C)$ with $f(c)\leq 0$, then we can find the desired $C'$ and $c'$ by the assumption of Case 1 and the IVP (\Cref{thm:ivp}), so we may assume $f(c)>0$ for every $c\in p(C)$. Let $C'$ be generated as an  $\mathbb R$-vector space by $C$ and a vector $b'\notin C$. Define $\sigma$ on $C'$ by $\sigma(c+\lambda b')\coloneqq \sigma(c)+\lambda(b'-a)$ for every $\lambda\in \mathbb R$ and $c\in C$ (here $\sigma(c)$ is calculated in the structure $C$). Clearly $\sigma\from C'\to C'$ is an automorphism of the $\mathbb R$-vector space $C'$. Define a linear order on $C'$ by declaring $b'<c$ for every $c\in C$ such that there is $a\in R_p$ with $a\leq c$, and $b'>c$ for all $c\in C$ for which there is no such $a$. We claim that $\sigma$ is an automorphism of the ordered $\mathbb R$-vector space $C'$. To this end, it is enough to check that if $b'<c$ with $c\in C$ then $\sigma(b')<\sigma(c)$ and that if $b'>c$ then $\sigma(b')>\sigma(c)$. Note first that the substructure $C''$ of $C'$ generated by $A\cup \{b'\}$ is a model of $\mathbb R\mathsf{-OVSA}$ isomorphic to $B$ over $A$ via $b'\mapsto b$.

Suppose $b'<c$. Then there is $a\in R_p$ with $a\leq c$. As $C'',C\models \mathbb R\mathsf{-OVSA}$, we have $\sigma(b')<\sigma(a)\leq c$, as required.

Now suppose $b'>c$. If $c\in p(C)$, then $\sigma(b')=b'-a> b'-a-f(c)=b'-a-c+\sigma(c)+a=b'-c+\sigma(c)>\sigma(c)$. If $c\notin p(C)$, then there is $a\in L_p$ with $c\leq a$, so $\sigma(b')>\sigma(a)\geq \sigma(c)$.

Hence $C'\models \mathbb R\mathsf{-OVSA}$, $ C'\supseteq C$, $f(b')=0$ and $b'\in p(C')$, so in particular $a_0<b'<a_1$, as desired.

\textbf{Case 2.} For every $a_0,a_1\in A$ with $a_0\in L_p$ and $a_1\in R_p$ we have $\{x\in A:f(x)\leq 0\}\cap [a_0,a_1]\neq \emptyset$.

Here we may assume $f(c)<0$ for every $c\in p(C)$, and we proceed as in Case 1, except that we declare $b'>c$ for every $c\in C$ such that there is $a\in L_p$ with $a\geq c$, and $b'<c$ for all $c\in C$ for which there is no such $a$. We only need to check that $\sigma$ defined on $C'$ as above preserves the order. This boils down to showing $\sigma(b')<\sigma(c)$ for $c\in p(C)$, as the other cases can be handled by using elements from $A$, exactly as in Case 1. As $f(c)<0$, we have $\sigma(b')=b'-a< b'-a-f(c)=b'-a-c+\sigma(c)+a=b'-c+\sigma(c)<\sigma(c)$, as required.

Clearly, at least one of Case 1 or Case 2 must happen, hence the proof of the lemma is finished.
\end{proof}

\begin{lemma}[Algebraic amalgamation]\label{lemma:algap}
  Suppose $A,B,C\models \mathbb R\mathsf{-OVSA}$, $A\subseteq B$, $A\subseteq C$, and every element of $B$ is $\sigma$-algebraic over $A$, that is, $\cl_B(A)=B$. Then $B$ and $C$ amalgamate over $A$.
  \end{lemma}
  \begin{proof}
  By a standard iteration argument, we may assume that  $B$ is generated over $A$ as an $\mathbb R\mathsf{-OVSA}$ by a single element $b$.  Let $f$ be a $\sigma$-polynomial such that $f(b)=a$ for some $a\in A$. Then, similarly as in the proof of \Cref{thm:ivp}, $f=f'_0\ldots f'_{n'-1}f_0\ldots f_{n-1}$ for some absolutely monotone  $f'_0,\ldots ,f'_{n'-1}$ and some non-absolutely-monotone $f_0,\ldots, f_{n-1}$ of degree $1$.  As the composition $f'_0\ldots f'_{n'-1}$ is absolutely monotone and $f'_0\ldots f'_{n'-1}(f_0\ldots f_{n-1}(b))=a\in A$, putting $a'\coloneqq f_0\ldots f_{n-1}(b)$ and $A'\coloneqq \langle A, a'\rangle _{\sigma}$ we get by \Cref{lemma:monot_AP} that $A'$ and $C$ amalgamate over $A$, so there is $A'\subseteq D\models \mathbb R\mathsf{-OVSA}$  such that $C$ embeds into $D$ over $A$. Now it is enough to show that  $B$ and $D$ amalgamate over $A'$, see \Cref{fig:algapdiag}.

  Put $a''\coloneqq f_{n-1}(b) $ and $A''\coloneqq \langle A', a''\rangle_{\sigma}$. As $f_0\ldots f_{n-2}(a'')\in A'$,  by induction on $n$ we may assume $A''$ and $D$ amalgamate over $A'$, so there is $A''\subseteq D'\models \mathbb R\mathsf{-OVSA}$ such that $D$ embeds into $D'$ over $A'$.
  It is enough to show that $D'$ and $B$ amalgamate over $A''$. Write $g(x)\coloneqq f_{n-1}(x)-a''$. Since each $f_i$ is of the form required by \Cref{lemma:sign_change}, we may apply the latter and obtain $D'\subseteq D''\models \mathbb R\mathsf{-OVSA}$ and $d''\in D''$ with $g(d'')=0$ and $d''$ having the same cut in $A''$ as $b$. As we also have $g(b)=0$, this clearly implies that $b\mapsto d''$ induces an embedding $B\to D''$ over $A''$, which completes the proof.
  \end{proof}
\begin{figure}
  \begin{center}
        \begin{tikzpicture}[scale=4.5]
  \node(a) at (0,0){$A$};
  \node(b) at (1,0.5){$B$};
  \node(c) at (1,-0.5){$C$};
  \node(a1) at (0.5,0){$A'$};
  \node(a2) at (1,0.25){$A''$};
  \node(d) at (1,-0.25){$D$};
  \node(d1) at (1.5,0) {$D'$};
  \node(d2) at (2,0) {$D''$};

\path[->, thick,  font=\scriptsize,>= angle 90]
(a)  edge  (b)
(a2) edge  (b)
(a)  edge  (a1)
(a1) edge  (d)
(a)  edge  (c)
(a1) edge  (a2)
(c)  edge (d)
(d)  edge (d1)
(d1) edge (d2)
(b)  edge (d2)
(a2) edge (d1)
;
\end{tikzpicture}
  \end{center}
  \caption{The diagram constructed in the proof of \Cref{lemma:algap}.}\label{fig:algapdiag}
\end{figure}
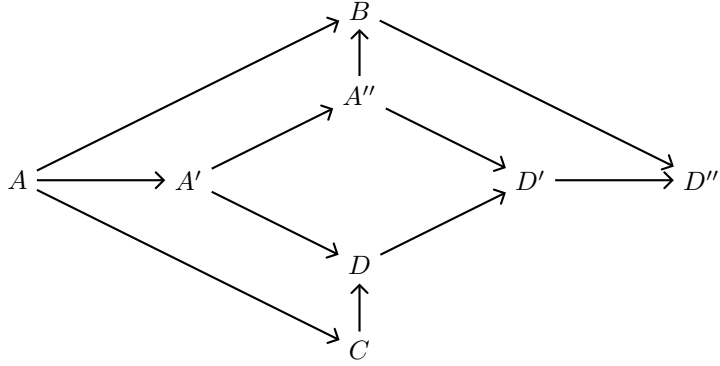

\subsection{Amalgamating $\sigma$-transcendental points}\label{subsec:tr_ap}
We briefly discuss the notion of a \emph{cell}, which we will use in \Cref{lemma:trans_AP}, and refer the reader to~\cite{van_den_dries_tame_1998} for details.

Roughly speaking, an $n$-dimensional cell (over $C$) in an ordered structure $A\supseteq C$ is a subset of $A^n$ obtained by inductively imposing for each $i<n$ a condition of the form $x_i=f_i(x_0,\ldots,x_{i-1})$ or $f_i(x_0,\ldots,x_{i-1})<x_i<g_i(x_0,\ldots,x_{i-1})$ where $f_i\from A^i\to A$ is a continuous\footnote{Here $A$ has the order topology and $A^i$ the induced product topology.} $C$-definable function or $f_i=-\infty$ and $g_i\from A^i\to A$ is a continuous $C$-definable function or $g_i=+\infty$. For a precise definition, see e.g.~\cite[Chapter 3, Definition~2.3]{van_den_dries_tame_1998}.

The \emph{Cell Decomposition Theorem} states that, if $A$ is an o-minimal structure, then every definable subset of $A^n$ is a finite disjoint union of cells.
Because the only functions definable in $\mathbb R\mathsf{-OVS}$ are piecewise $\mathbb R$-affine, we may refine each cell decomposition to one in which every cell is given by affine equations and affine linear inequalities. If such a cell is not open\footnote{An open cell is one in whose definition only inequalities are used (and no conditions of the form $x_i=f(x_0,\ldots ,x_{i-1})$). Equivalently, it is open in the aforementioned topology.}, then all of its elements satisfy a non-trivial affine equation.
\begin{defin}
We call $A\models\mathbb R\mathsf{-OVSA}$  \emph{coinitial} in  $B\models\mathbb R\mathsf{-OVSA}$ iff for every positive $b\in B$ there is a positive $a\in A$ with $a<b$.
\end{defin}
\begin{rem}\label{rem:aciecic}
By \Cref{rem:aspfp} and the fact that fixed points are $\sigma$-algebraic, if $M$ is existentially closed and $A=\cl_M(A)$, then $A$ is coinitial in $M$.
\end{rem}
\begin{lemma}\label{lemma:trans_AP}
Let $A,B,C \models \mathbb R\mathsf{-OVSA}$ with $A\subseteq B,C$ and $B=\langle A,b\rangle_{\sigma}$ for some $b\in B$, and assume $A$ is $\sigma$-algebraically closed in some existentially closed $M\supseteq B$. Then $B$ and $C$ amalgamate over $A$.
\end{lemma}
\begin{proof}
If $b$ is $\sigma$-algebraic over $A$ then we are done by algebraic amalgamation (\Cref{lemma:algap}), so suppose $b$ is $\sigma$-transcendental over $A$. 
By compactness it is enough to show that if $\phi(x)$  is a quantifier-free $L_{\mathbb R, \sigma}$-formula over $A$ which is satisfied by $b$, then it has a realisation in $A$, as then by \Cref{fact:diagram} we may embed $B$ in an extension of $C$. 
As each such formula $\phi(x)$ is equivalent to a formula $\psi(x,\sigma(x),\ldots,\sigma^k(x))$ for some $k<\omega$ and some $L_{\mathbb R}$-formula $\psi(x_0,x_1,\ldots,x_k)$,
 by o-minimality of $\mathbb R\mathsf{-OVS}$, we may assume $\psi$ defines an $\mathbb R\mathsf{-OVS}$  cell over $A$, say $Z$. As $(b,\sigma(b),\ldots,\sigma^k(b))\in Z$  and $b$ is $\sigma$-transcendental over $A$, the cell $Z$ must be open, hence it is defined by a conjunction of inequalities of the form $\ell(x_0,x_1,\ldots,x_k,a)>0$ with $a\in A$ and $\ell$ an $\mathbb{R}$-affine function. 
Thus $\phi(x)$ is equivalent to $\bigwedge_{i=0}^{n-1} f_i(x)>0$ for some $\sigma$-polynomials $f_0,\ldots,f_{n-1}$ over $A$. Put $h(x)\coloneqq \min_{i<n}f_i(x)$. As $b\models \phi$, we have $h(b)>0$, so by \Cref{rem:aciecic} there is some $a\in A$ with $0<a<h(b)$. Let $f^a_i\coloneqq f_i-a$ and $h^a\coloneqq h-a=\min_{i<n}f^a_i$. 

If $h^a$ were negative on every element of $A$ then, as $h^a$ has the IVP on $M$ by \Cref{thm:minima} and $h^a(b)>0$, we could find $b_1\in M$ with $h^a(b_1)=0=f^a_i(b_1)$ for some $i$, which, by $\sigma$-algebraic closedness of $A$ in $M$ would give $b_1\in A$, a contradiction. Thus there is an element $a_1\in A$ with $h^a(a_1)\geq 0$, so  $h(a_1)\geq a>0$ and so $a_1\models \phi$.
\end{proof}

\begin{thm}\label{thm:apqtor}
The category of ordered abelian groups with an automorphism has the Amalgamation Property, and so do the category of models of $\mathbb R\mathsf{-OVSA}$ and the category of models of $\mathbb Q\mathsf{-OVSA}$.
\end{thm}
\begin{proof}
By \Cref{rem:oagstoqovs} and \Cref{pr:Rreduction}, it is enough to prove the AP for models of $\mathbb R\mathsf{-OVSA}$. So let $A,B,C\models \mathbb R\mathsf{-OVSA}$ with $A\subseteq B,C$; we may assume $B=\langle A,b\rangle_{\sigma}$ for some $b\in B$. Let $M\supseteq B$ be existentially closed, and let $A'\coloneqq \cl_M(A)$. By \Cref{rem:pregeometry}, $A'$ is $\sigma$-algebraically closed in $M$.
 Put $B'\coloneqq \langle A',b\rangle_{\sigma}\subseteq M$. By algebraic amalgamation (\Cref{lemma:algap}) applied to $A\subseteq A', C$, there is $C'\supseteq A'$ such that $C$ embeds in $C'$ over $A$, as in \Cref{fig:bkdiag}. As $M$ is existentially closed and $A'$ is $\sigma$-algebraically closed in $M$, we may apply \Cref{lemma:trans_AP}  to $A'\subseteq B', C'$ and find $D\supseteq B'\supseteq B$ such that $C'$ embeds in $D$ over $A'$. As $C$ embeds in $C'$ over $A$, it follows that $C$ embeds in $D$ over $A$, thus $D$ solves the amalgamation problem for $A\subseteq B,C$.
\end{proof}
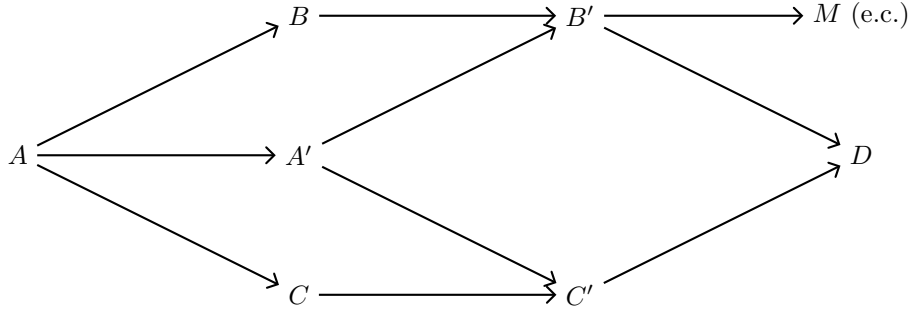
\begin{figure}
  \begin{center}
        \begin{tikzpicture}[scale=3.7]
  \node(a) at (0,0){$A$};
  \node(b) at (1,0.5){$B$};
  \node(c) at (1,-0.5){$C$};
  \node(a1) at (1,0){$A'$};
  \node(a1b) at (2,0.5){$B'$};
  \node(b1) at (3,0.5){$M$ (e.c.)};
  \node(c0) at (2,-0.5){$C'$};
  \node(d) at (3,0) {$D$};

\path[->, thick,  font=\scriptsize,>= angle 90]
(a)   edge (b)
(b)   edge (a1b)
(a1b) edge (b1)
(a)   edge (a1)
(a1)  edge (a1b)
(a1)  edge (c0)
(a)   edge (c)
(c)   edge (c0)
(c0)  edge (d)
(a1b) edge (d)
;
\end{tikzpicture}
  \end{center}
  \caption{The diagram constructed in the proof of \Cref{thm:apqtor}.}\label{fig:bkdiag}
\end{figure}
\begin{rem}
As typical with categories of linearly ordered structures, neither of the categories considered above has pushouts: if $B$, $C$ are arbitrary nontrivial, then no fixed amalgam of $B, C$ over $\set 0$ admits embeddings into $B\times_\mathrm{lex} C$ and into $C\times_\mathrm{lex} B$ commuting with the embeddings of $B$ and $C$.
\end{rem}

The following example shows that the category of ordered abelian groups equipped with two unrelated automorphisms, that is, equipped with an action of the free group $F_2$ by automorphisms, does not have the Amalgamation Property.
\begin{eg}
Take any $(A,\sigma_0)\models \mathbb Q\mathsf{-OVSA}$ such that some absolutely monotone $\sigma_0$-polynomial $f(x)$  with rational coefficients that does not have a root in $A$, e.g.~as in \Cref{eg:norootofabsmon}. Let $b$ be the unique root of $f(x)$ in some existentially closed extension $(M,\sigma_0)\supseteq (A,\sigma_0)$ and, up to enlarging $M$, fix  $\epsilon\ll \seq{A, b}_{\sigma_0}$. Let $B\coloneqq\seq{A, b, \epsilon}_{\sigma_0}$, and observe it is closed under $\sigma_0$. Define $\sigma_1$ on $B$ as $\sigma_1(a+\lambda b+\mu \epsilon)\coloneqq a+\lambda b+(\mu+\lambda) \epsilon$, which is easily shown to be an automorphism of $B$ as an ordered $\mathbb Q$-vector space. Then $(B, \sigma_0, \mathrm{id})$ does not amalgamate with $(B, \sigma_0, \sigma_1)$ over $(A, \sigma_0, \mathrm{id})$.
\end{eg}
\section{Dependent positive theories}\label{sec:posnip}
In this section we study dependent (or $\mathsf{NIP}$) positive theories, which we introduce in  \Cref{def:nip}.  For the sake of brevity, we assume the reader to be familiar with classical model theory, in particular with notions such as that of \emph{indiscernible sequence}. We refer to~\cite{hodges} for basic model theory, and to~\cite{simon_2015} for $\mathsf{NIP}$ theories in the classical sense.
\subsection{A tale of homomorphisms}\label{sec:poslogprelim}
We briefly recall the basics of positive logic, following the foundational work in \cite{BY,BYP}. Note that our terminology differs slightly from that used there, in particular we use the term \emph{positively existentially closed model} instead of \emph{existentially closed model} in order to avoid confusion. From now on, lowercase letters such as $a,b\ldots$  denote tuples of elements of a model; their length is denoted by $\abs a, \abs b\ldots$. Similarly for tuples $x,y,\ldots$ of variables. Formulas are over $\emptyset$ unless otherwise stated, or unless parameters are made explicit, as in $\phi(x,b)$.

\begin{defin}
\label{def:positve-syntax}
Fix a signature $L$. A \emph{positive formula} in $L$ is one obtained by combining atomic formulas using $\wedge$, $\vee$, $\top$, $\bot$, and $\exists$. An \emph{h-universal sentence} is a negation of a positive sentence. An \emph{h-inductive sentence} is a sentence of the form $\forall x\;(\phi(x) \to \psi(x))$, where $\phi(x)$ and $\psi(x)$ are positive. A \emph{positive theory} is a consistent set of h-inductive sentences. 
\end{defin}
As in full first-order logic, we will assume that $L$ contains a symbol $=$ interpreted in every $L$-structure as equality.

Unless otherwise stated, from now on $\phi$, $\psi$, etc.~will denote positive formulas.
\begin{rem}
\label{rem:morleyisation}
Full first-order logic can be studied as a special case of positive logic through \emph{Morleyisation}, see e.g.~\cite[Remark 2.2]{DK}.
\end{rem}

\begin{defin}
\label{def:homomorphism}
A function $f\from A \to B$ between $L$-structures is called a \emph{homomorphism} iff for every atomic formula (equivalently, every positive formula) $\phi(x)$ and every $a \in A$ we have
\[
A \models \phi(a) \implies B \models \phi(f(a)).
\]
We call $f$ an \emph{immersion} (resp.~embedding) iff additionally the converse implication holds for all positive (resp.~all atomic, equivalently all quantifier-free) $\phi(x)$ and all $a \in A$. The codomain $B$ of a homomorphism $f\from A\to B$ is called a \emph{continuation} of $A$.
\end{defin}
\begin{defin}
\label{def:ec-model}
We call a model $M$ of a positive theory $T$ a \emph{positively existentially closed model}, or a \emph{pec model}, iff the following equivalent\footnote{The implications $\text{\ref{point:ecm1}}\Leftrightarrow\text{\ref{point:ecm2}}$ and $\text{\ref{point:ecm3}}\allora \text{\ref{point:ecm1}}$ are immediate,  $\text{\ref{point:ecm1}}\allora \text{\ref{point:ecm3}}$ is~\cite[Lemme~14]{BYP}.} conditions hold.
\begin{enumerate}[label=(\roman*)]
\item\label{point:ecm1} Every homomorphism $f\from M \to A$ with $A \models T$ is an immersion.
\item\label{point:ecm2} For every $a \in M$ and positive formula $\phi(x)$ such that there are $A\models T$ and a homomorphism $f\from M \to A$ with $A \models \phi(f(a))$, we have that $M \models \phi(a)$.
\item \label{point:ecm3}For every $a \in M$ and  positive formula $\phi(x)$ such that $M \not \models \phi(a)$ there is a positive formula $\psi(x)$ with $T \models \neg \exists x (\phi(x) \wedge \psi(x))$ and $M \models \psi(a)$.
\end{enumerate}
We will usually denote pec models of a positive theory $T$ by $M,N$, and arbitrary models and arbitrary parameter sets by $A,B,C,D$.
\end{defin}
\begin{fact}
\label{fact:positive-logic-facts}
Let $T$ be a positive theory.
\begin{enumerate}[label=(\roman*)]
\item \emph{(Direct limits)} The limit of a direct system of (pec) models is a (pec)~model.
\item \emph{(Amalgamation)} If $f_0\from A\to B_0$, $f_1\from A\to B_1$ are homomorphisms and at least one of $f_0$,$f_1$ is an immersion, then there are a model $C$ and $g_0\from B_0\to C$ and $g_1\from B_1\to C$ such that $g_0f_0=g_1f_1$. In particular, every pec model is an \emph{amalgamation base}, that is, every amalgamation problem as above with $A$ pec has a solution.
\item \emph{(Pec continuation)} For every $A \models T$ there is are a pec $M\models T$ and a homomorphism $f\from A \to M$.
\item \emph{(Compactness)} Let $\Sigma(x)$ be a set of positive formulas and suppose that for every finite $\Sigma_0(x) \subseteq \Sigma(x)$ there are $A \models T$ and $a \in A$ such that $A \models \Sigma_0(a)$. Then there are a pec  $M\models T$ and $a \in M$ such that $M \models \Sigma(a)$.
\end{enumerate}
\end{fact}

\begin{defin}
\label{def:jcp}
We say that a positive theory $T$ has the \emph{Joint Continuation Property} or \emph{JCP} if the following equivalent\footnote{The implication $\neg\text{\ref{point:jcp2}}\allora\neg\text{\ref{point:jcp1}}$ is straightforward, while $\text{\ref{point:jcp2}}\allora\text{\ref{point:jcp1}}$ follows from a compactness argument using the positive version of \Cref{fact:diagram}, namely that continuations of $A$ are the same as models of the positive quantifier-free $L(A)$-sentences true in $A$.} conditions hold.
\begin{enumerate}[label=(\roman*)]
\item\label{point:jcp1} For every two models $A_1$ and $A_2$ there are $B\models T$ and homomorphisms $A_1 \to B \leftarrow A_2$.
\item\label{point:jcp2} If $T \models \neg \phi \vee \neg \psi$, where $\phi$ and $\psi$ are positive sentences, then $T \models \neg \phi$ or $T \models \neg \psi$.
\end{enumerate}
\end{defin}
The JCP is a counterpart of completeness: for \emph{h-universal} theories, that is, theories axiomatised by h-universal sentences, it corresponds to being the h-universal theory of some structure, see~\cite[Lemme~18]{BYP}.
\begin{defin}
\label{def:saturated}
Let $M$ be a pec model of a positive theory $T$. We say that $M$ is \emph{$\kappa$-saturated} if whenever $A \subseteq M$  is such that  $\abs{A} < \kappa$ and $\Sigma(x)$ is a set of positive formulas over $A$, then $\Sigma(x)$ is satisfiable in $M$ if and only if it is finitely satisfiable in $M$.
\end{defin}
\begin{fact}
\label{fact:saturated-models}
For every $A\models T$ and every $\kappa \geq \abs{A} + \abs{T}$ there is a $\kappa^+$-saturated continuation $N$ of $A$ with $\abs{N} \leq 2^\kappa$.
\end{fact}

Fix a sufficiently large cardinal $\bar{\kappa}$. We will say a set is $\emph{small}$ if it is of cardinality smaller than $\bar{\kappa}$.
\begin{ass}
\label{conv:monster-model}
Until the end of the paper, unless otherwise stated, we assume that $T$ is a positive theory with the JCP, so we can (and will) work in a \emph{monster model} $\monster$ (sometimes also called a universal domain), that is:
\begin{itemize}
\item \emph{positively existentially closed}: $\monster$ is a pec model;
\item \emph{very homogeneous}: every partial immersion $f\from \monster \to \monster$ with small domain extends to an automorphism  of $\monster$;
\item \emph{very saturated}: every finitely satisfiable small set of positive formulas over $\monster$ is satisfiable in $\monster$.
\end{itemize}
We will assume all parameter sets to be small, except when we consider the monster model or a superset of it as a parameter set.
\end{ass}
We will write $\tp(a/B)$ for the set of all positive formulas over $B$ satisfied by $a$. So we have $\tp(a/B) = \tp(a'/B)$ if and only if there is an automorphism $f\from \monster \to \monster$ fixing $B$ pointwise such that $f(a) = a'$. We also write $a \equiv_B a'$ in this case.
By a \emph{maximal type} (over $A$) in $T$ we will mean a maximal consistent with $T$ set of positive formulas (over $A$). Equivalently, it is the set of positive formulas satisfied by a tuple \emph{of the monster model}.\footnote{In particular, it is a tuple from a pec model, cf.~\Cref{rem:primetypes}.} By a partial type (over $A$) in $T$ we will mean any consistent set of positive formulas (over $A$). We sometimes refer to types as \emph{positive types}. 
From now until \Cref{rem:primetypes}, we will exclusively work with maximal types (unless we are talking of partial types).

If we call a sequence \emph{indiscernible}, we mean with respect to positive formulas (over $\emptyset$, if no set of parameters is mentioned).
\begin{defin}
\label{def:indiscernible-sequence-based-on-long-sequence}
We say that a sequence $(b_j)_{j \in J}$ is \emph{based on} a sequence $(a_i)_{i \in I}$ \emph{over $A$} iff  for all $\bla j1<n\in J$ there are $i_1 < \ldots < i_n \in I$ with $b_{j_1} \ldots b_{j_n} \equiv_A a_{i_1} \ldots a_{i_n}$.
\end{defin}
The following appears as~\cite[Lemma 3.1]{pillay_forking_2000} and~\cite[Lemma 1.2]{ben-yaacov_simplicity_2003}.
\begin{fact}\label{fact:extracting}
Let $A$ be a parameter set, $\kappa$ a cardinal, and let $\lambda = \beth_{(2^{\abs{T} + \abs{A} + \kappa})^+}$. Then for every sequence $(a_i)_{i < \lambda}$ of $\kappa$-tuples there is an $A$-indiscernible sequence $(b_i)_{i < \omega}$ based on $(a_i)_{i < \lambda}$.
\end{fact}

\subsection{Pillay and Robinson theories}\label{sec:pilrob}

In this subsection we take a look at how to code classical inductive theories in positive logic.

\begin{defin}
  A positive theory is \emph{Pillay} iff the negation of every atomic formula is equivalent modulo $T$ to some positive quantifier-free formula.
\end{defin}
\begin{rem}\label{rem:pillayqfpos}
  By an easy induction on formulas, we see that in a Pillay theory every quantifier-free formula is equivalent to a positive quantifier-free one.
\end{rem}
\begin{rem}\label{rem:pilparmor}
Pillay theories, named after \cite{pillay_forking_2000}, may be identified with classical inductive theories, and vice-versa. The pec models of Pillay theories  correspond to the existentially closed models of classical inductive theories in the usual  sense\footnote{We recalled a special case in~\Cref{defin:ec}.}. Moreover, every homomorphism between models of a Pillay theory is an embedding.
\end{rem}
\begin{proof}
  Let $T$ be Pillay. In particular, the (h-inductive) axioms of $T$ are inductive sentences. Moreover, every homomorphism between models of $T$ is easily checked to be an embedding.

  Conversely, given a classical inductive $L'$-theory $T'$, we may view it as a Pillay positive theory by a partial Morleyisation.  Namely, for each atomic $L'$-formula $\phi(x)$, we add a predicate symbol $R_{\neg\phi}$ to the language. Call the resulting language $L$,   and consider the h-inductive axioms $\forall x\; \neg (\phi(x)\wedge R_{\neg\phi}(x))$ and $\forall x\; (\phi(x)\vee R_{\neg\phi}(x))$, where $\phi(x)$ is atomic. By negation normal form, modulo these axioms every quantifier-free formula is equivalent to a positive quantifier-free one. Now translate the (inductive) axioms of $T'$ into h-inductive $L$-sentences and add them to the axioms above, obtaining a Pillay theory $T$. Models of $T$ coincide with the obvious $L$-expansions of models of $T'$, every quantifier-free $L'$-formula is identified with a positive quantifier-free $L$-formula, and $L$-homomorphisms are the same as $L'$-embeddings.

  The conclusion follows by the definitions of pec model and of existentially closed model.
\end{proof}

\begin{defin}
  If $T$ is a Pillay theory, we denote by $T_\mathrm{sub}$ the theory consisting of all consequences of $T$ of the form
  \begin{enumerate}
  \item $\forall x\; \phi(x)$, for $\phi$ positive quantifier-free, or
  \item $\forall x\; \neg \phi(x)$, for $\phi$ positive quantifier-free.
  \end{enumerate}
\end{defin}

The second clause in the definition above is indeed needed: even if $\neg\phi(x)$ is equivalent to a positive quantifier-free formula in $T$, this is not true without suitable axioms. We will show below that those in $T_\mathrm{sub}$ suffice.

\begin{rem}\label{rem:huintf}\*
  \begin{enumerate}
  \item Let $T$ be Pillay, let $\phi(x)$ be an atomic formula, and let $\psi(x)$ be a positive quantifier-free formula equivalent to $\neg \phi(x)$ modulo $T$. Then, $T_\mathrm{sub}$ contains the axioms $\forall x\; (\phi(x)\vee \psi(x))$ and $\forall x\; \neg(\phi(x)\wedge \psi(x))$. Therefore, $\psi(x)$ is equivalent to $\neg \phi(x)$ modulo $T_\mathrm{sub}$ as well.
  \item In particular, $T_\mathrm{sub}$ is Pillay, and \Cref{rem:pillayqfpos} applies to it.
  \item It follows that $T_\mathrm{sub}$ is (equivalent to) the set of universal consequences, in the classical sense, of models of $T$; equivalently, it is the common theory of substructures of models of $T$ (see e.g.~\cite[Corollary~6.5.3]{hodges}). If we view $T$ as a classical inductive theory, then $T_\mathrm{sub}$ is usually denoted by $T_\forall$. We prefer to avoid this notation for the reason below.
  \item The theory $T_\mathrm{sub}$ is not the same as the set of h-universal consequences of $T$ (which is sometimes denoted by $T_\forall$ as well) since, for positive $\phi(x), \psi(x)$, the sentence  $\forall x\; (\phi(x)\vee \psi(x))$ is not h-universal. Note that the partial Morleyisation in the proof of \Cref{rem:pilparmor} requires sentences of the latter form.
  \end{enumerate}
\end{rem}
If $T$ is Pillay, amalgamation of embeddings between models of $T_\mathrm{sub}$ corresponds to a form of infinitary quantifier elimination, \Cref{fact:robinfqe} below. By contrast, amalgamation of homomorphisms between models of the h-universal consequences of $T$ has quite more drastic consequences, see~\cite[Proposition~9]{poizat_positive_2018}.
\begin{defin}
A \emph{Robinson theory} is a Pillay theory such that the category of models of $T_\mathrm{sub}$ with embeddings as arrows has the Amalgamation Property.
\end{defin}
\begin{rem}\label{rem:rovsarob}
  The theory of ordered abelian groups with an automorphism, and the theories $\mathbb Q\mathsf{-OVSA}$ and $\mathbb R\mathsf{-OVSA}$ studied in the previous sections, may all be viewed as positive theories which are Pillay, as $\neg (x<y)$ is equivalent in them to $(x=y)\vee (x>y)$, and $\neg (x=y)$ is equivalent to $(x<y)\vee (x>y)$. Moreover, \Cref{thm:apqtor} shows they are Robinson.
\end{rem}

\begin{fact}\label{fact:robinfqe}
  Let $T$ be Robinson. For all positive formulas $\phi(x)$, there is a collection of positive quantifier-free formulas $(\phi_i(x))_{i\in I}$ such that $\phi(x)$ is equivalent to the infinitary conjunction $\bigwedge_{i\in I}\phi_i(x)$ in every pec model of $T$.
\end{fact}
\begin{proof}
See e.g.~\cite[Theorem 8.5.2 and Theorem 8.5.4]{hodges}.
\end{proof}

\subsection{Generalising NIP}\label{sec:nipbasics}
Recall that $T$ denotes a positive theory with the JCP, and $\monster$ a monster pec model.
We start by introducing a notion of independence property for positive logic, which extends the usual independence property in full first-order logic.

\begin{defin}\label{def:nip}
Fix a positive theory $T$ with the JCP, and let $\phi(x;y)$ be a positive formula.\footnote{Strictly speaking, we should talk of \emph{partitioned formulas}, that is, formulas with free variables partitioned in ``object'' and ``parameter'' variables.} We say that $\phi$ has the \emph{Independence Property} \emph{$\mathsf{IP}$} iff there are a positive $\psi(x;y)$ and tuples $(a_i)_{i\in\omega}$, $(b_W)_{W\subseteq \omega}$ from some $A\models T$ such that $T\models \forall x, y\; \neg (\phi(x;y)\wedge \psi(x;y))$ and
  \begin{gather*}
i\in W\then    A\models \phi(a_i;b_W)  \\
i\notin W\then    A\models \psi(a_i;b_W)  
\end{gather*}
Otherwise, we say that $\phi$ is \emph{$\mathsf{NIP}$}. We say that $T$  is \emph{$\mathsf{NIP}$}, or \emph{dependent}, iff every positive $\phi$ is $\mathsf{NIP}$.
\end{defin}

\begin{rem}\label{rem:pecnipenough}
  If we can find, in a model $A$ of $T$, tuples witnessing that a formula has $\mathsf{IP}$, then by continuing $A$ to a pec $M$ we can find such tuples in a pec $M\models T$. So, in order to check that $T$ is $\mathsf{NIP}$, it is enough to do so for pec models.
\end{rem}
Recall that a positive theory is \emph{bounded} iff there is an absolute bound on the cardinality of its pec models.  
\begin{rem}\label{rem:canstretch}
By compactness, if $T$ has $\mathsf{IP}$, then it satisfies \Cref{def:nip} with $\omega$ replaced by any cardinal. In particular $T$ cannot be bounded because the $a_i$ need to be pairwise distinct.
\end{rem}
\begin{defin}[Alternation number]
  Let $\phi(x;y)$ and $\psi(x;y)$ be mutually contradictory positive formulas.
  \begin{enumerate}
  \item Let $b$ be a tuple of parameters and $(a_i)_{i<\omega}$ an indiscernible sequence. 
    If there is a maximum $n\in \omega$ such that there exist $i_0<\ldots<i_n$ with $\models \phi(a_{i_j};b)$ for even $j\leq n$ and $\models \psi(a_{i_j};b)$ for odd $j\leq n$, we define $\operatorname{alt}(\phi,\psi,(a_i), b)=n$. In the case where there is no maximum such $n$,  we put $\operatorname{alt}(\phi,\psi,(a_i), b) = \infty$.
\item  We set $\operatorname{alt}(\phi,\psi)\coloneqq\sup\set{\operatorname{alt}(\phi,\psi,(a_i), b): (a_i)_{i<\omega}, b\in \monster}$.
\end{enumerate}
\end{defin}
\begin{lemma}\label{lem:alt}
Let $\phi(x;y)$ and $\psi(x;y)$ be mutually contradictory positive formulas. Then $\phi(x;y)$ has $\mathsf{IP}$ witnessed by $\psi(x;y)$ if and only if $\operatorname{alt}(\phi,\psi)=\infty$. 

Moreover, if these equivalent conditions hold, then there are $(a_i)_{i<\omega}$ and $b$ such that  $\operatorname{alt}(\phi,\psi,(a_i), b)=\infty$.
\end{lemma}
\begin{proof}
  $\Leftarrow$: Note that, by compactness, if $\operatorname{alt}(\phi,\psi)=\infty$ then there is some $b$ and a sequence $(a_i)_{i<n}$ such that $\models \phi(a_i;b)$ for every even $i$ and $\models \psi(a_i;b)$ for every odd $i$. Now proceed as in e.g.~\cite[proof of Lemma 2.7]{simon_2015}.
  
$\Rightarrow$: For every cardinal $\lambda$ we can find by compactness a sequence $(a_i)_{i<\lambda}$ such that for every $W\subseteq \lambda$ the set $\pi_W\coloneqq \set{\phi(a_i;y):i\in W}\cup \set{\psi(a_j;y):j\in \lambda \setminus W}$ is consistent. By \Cref{fact:extracting}, we may assume that $(a_i)_{i<\omega}$ is indiscernible. Let $b\models \pi_{\{2i:i\in \omega\}}$. Then $(a_i)_{i<\omega}$ and $b$ witness that $\operatorname{alt}(\phi,\psi)=\infty$. This also proves the ``moreover'' clause.
\end{proof}

\begin{rem}\label{rem:altpar}
If there is $(\phi,\psi,(a_i), b)$ with infinite alternation number, then $(\phi',\psi',(a_i), b')$ has infinite alternation number for some $\phi',\psi'$ over $\emptyset$ and some $b'$.
\end{rem}
\begin{proof}
  Suppose $\phi(x;y,d)\wedge\psi(x;y,d)$ is inconsistent. Then, since $d$ is a tuple from a pec model, there is $\theta\in \tp(d/\emptyset)$ such that $\theta(z)\wedge (\phi(x;y,z)\wedge\psi(x;y,z))$ is inconsistent. It is then sufficient to set $\phi'(x;y,z)\coloneqq\phi(x;y,z)\wedge \theta(z)$,  $\psi'(x;y,z)\coloneqq\psi(x;y,z)\wedge \theta(z)$ and $b'\coloneqq (b,d)$.
\end{proof}

\begin{pr}\label{pr:lor}
  If $\phi_0$ and $\phi_1$ are $\mathsf{NIP}$ then so is $\phi_0\vee\phi_1$.
\end{pr}
\begin{proof}
  Suppose not, as witnessed by $\psi(x;y)$ inconsistent with $(\phi_0\vee \phi_1)(x;y)$ and $(a_i)_{i<\omega}, b$ such that $\models (\phi_0\vee \phi_1)(a_i;b)$ for every even $i$ and $\models \psi(a_i;b)$ for every odd $i$. For some $j<2$ we must have that $\phi_j(a_i;b)$  is true for infinitely many even $i$. But $\phi_j\wedge \psi$ is still inconsistent, so $\psi$ also witnesses that $\phi_j$ has $\mathsf{IP}$.
\end{proof}
 
\begin{pr}\label{pr:land}
  If $\phi_0, \phi_1$ are $\mathsf{NIP}$ then so is $\phi_0\wedge\phi_1$.  
\end{pr}
\begin{proof}
  Suppose not; then there are some positive $\psi$ inconsistent with $\phi_0\wedge \phi_1$, some indiscernible $(a_i)_{i<\abs T^+}$ and some $b$ such that $\phi_0\wedge \phi_1(a_{2i};b)$ and $\psi(a_{2i+1}; b)$ hold for every $i$.  If $\phi_0(a_{2i+1}; b)$ holds for infinitely many $i<\abs T^+$, then $\phi_0\wedge \psi$ witnesses that $\phi_1$ has $\mathsf{IP}$. Otherwise, up to discarding finitely many $i$, for each $i<\abs T^+$ there is a positive  $\theta_i$ such that $\theta_i(x; b)\wedge \phi_0(x; b)$ is inconsistent and $a_{2i+1}\models \theta_i(x;b)$. By the pigeonhole principle there is a positive $\theta$ such that for infinitely many $i$ we have $\theta_i=\theta$. Now $\psi\wedge \theta$ witnesses that $\phi_0$ has $\mathsf{IP}$.
\end{proof}
Recall that, for $\phi(x;y)$ a partitioned formula, $\phi^{\mathrm{op}}(y;x)$ denotes the same formula but with the opposite partition, i.e.~with $y$ the tuple of object variables and $x$ the tuple of parameter variables.
\begin{pr}
A positive formula   $\phi$ has $\mathsf{IP}$ witnessed by $\psi$ if and only if $\phi^{\mathrm{op}}$ has $\mathsf{IP}$ witnessed by $\psi^{\mathrm{op}}$.
\end{pr}
\begin{proof}
Essentially the same as in~\cite[Lemma~2.5]{simon_2015}.
\end{proof}

\begin{pr}\label{pr:nip1var}
If all positive formulas $\phi(x;y)$ with $\abs y=1$ are $\mathsf{NIP}$, then $T$ is $\mathsf{NIP}$.
\end{pr}
\begin{proof}
  We closely follow the proof of~\cite[Proposition~2.11]{simon_2015}.
  Assume that all formulas $\phi(x; y)$ with $\abs{y} = 1$ are $\mathsf{NIP}$. 
  \begin{claim}\label{claim:tailindoverb}
  Let $(a_i)_{i<\abs{T}^+}$ be an indiscernible sequence of tuples, and
let $b\in \monster$ with $\abs{b}=1$. Then there is some $\alpha<\abs{T}^+$ such that the sequence $(a_i)_{\alpha<i<\abs{T}^+}$ is indiscernible over $b$.
  \end{claim}
  \begin{claimproof}
For every pair $(\phi(x_0,\ldots,x_n;y),\psi(\phi(x_0,\ldots,x_n;y))$ of mutually contradictory formulas,   by \Cref{lem:alt} there is some $\alpha_{(\phi,\psi)}<\abs{T}^+$ such that  for every $\alpha_{(\phi,\psi)}<i_0<\ldots<i_n<\abs{T}^+$ and $\alpha_{(\phi,\psi)}<j_0<\ldots<j_n<\abs{T}^+$ we cannot have that $\models \phi(a_{i_0},\ldots,a_{i_n};b)$ and $\models \psi(a_{j_0},\ldots,a_{j_n};b)$ (otherwise we can inductively find an infinite sequence of $n+1$-tuples of elements of the sequence $(a_i)$ on which $\phi(x_0,\ldots,x_n;b),\psi(x_0,\ldots,x_n;b)$ alternate infinitely many times).
  Put $\alpha\coloneqq \sup_{\phi\wedge\psi\implica \bot} \alpha_{(\phi,\psi)}$ and observe that $\alpha<\abs T^+$ because the latter is a regular cardinal. Then $(a_i)_{\alpha<i<\abs{T}^+}$ is indiscernible over $b$: if not, then there is a positive formula $\phi(x_0,\ldots,x_n;y)$ such that  $\models \phi(a_{i_0},\ldots,a_{i_n};b)$ but $\models \neg \phi(a_{j_0},\ldots,a_{j_n};b)$ for some $\alpha<i_0<\ldots<i_n$ and $\alpha<j_0<\ldots<j_n$.  By condition~\ref{point:ecm3} in \Cref{def:ec-model} there is a positive formula $\psi(x_{0},\ldots,x_{n};y)$ contradictory with $\phi(x_{0},\ldots,x_{n};y)$ such that $\models \psi(a_{j_0},\ldots,a_{j_n};b)$. But, as $\alpha_{(\phi,\psi)}<\alpha<i_0<\ldots<i_n$ and $\alpha_{(\phi,\psi)}<\alpha<j_0<\ldots<j_n$, this contradicts the choice of $\alpha_{(\phi,\psi)}$.
\end{claimproof}
 The conclusion follows from \Cref{claim:tailindoverb} exactly as in \cite{simon_2015}.
\end{proof}
As pointed out in~\cite[Theorem 10.18]{DK}, the technique used there can be used to prove that adding hyperimaginary sorts preserves $\mathsf{NIP}$ (in positive logic). We verify the details below; for background on the hyperimaginary extension $T^{\mathrm{heq}}$ considered in positive logic see~\cite[Example 2.16]{BY} or~\cite[Subsection 10.3]{DK}. $\mathsf{NIP}$ (along with stability) of hyperdefinable sets was extensively studied in  \cite{Haskel,KP1,KP2}.

\begin{pr}
  If $T$ is $\mathsf{NIP}$, then  $T^{\mathrm{heq}}$ is also $\mathsf{NIP}$.
\end{pr}

\begin{proof}
Suppose $T^{\mathrm{heq}}$ has $\mathsf{IP}$, witnessed by a pair of mutually contradictory positive formulas $\phi(x;y)$ and $\psi(x;y)$ and parameters $([a_i])_{i<\omega}$ and $([b_W])_{W\subseteq \omega}$. Here $x$ and $y$ are hyperimaginary variables (in contrast to \cite{DK}, here we omit variables of the home sort in our notation by identifying the home sort with its quotient by the equality relation), and  $[a_i]$ indicates the equivalence class of $a_i$ modulo the type-definable equivalence relation corresponding to the sort $x$, and similarly for $[b_W]$.
Let $x_r$ and $y_r$ be home-sort variables corresponding to $x$ and to $y$, respectively. By~\cite[Lemma 10.10]{DK}  there are partial positive types $\Sigma_\phi(x_r;y_r)$ and $\Sigma_\psi(x_r;y_r)$ such that $\models \phi([a_i]; [b_W])\iff \models \Sigma_\phi(a_i; b_W)$, and similarly for $\psi$ and $\Sigma_\psi$.
Then $\Sigma_\phi(x_r;y_r)\cup\Sigma_\psi(x_r;y_r)$ is inconsistent, so there are $\phi'\in \Sigma_\phi$ and $\psi'\in \Sigma_\psi$ such that $\phi'(x_r;y_r)\wedge \psi'(x_r;y_r)$ is inconsistent. We claim that the formulas $\phi'(x_r;y_r)$ and $\psi'(x;y)$ together with the parameters $(a_i)_{i<\omega}$ and $(b_W)_{W\subseteq \omega}$ witness that $T$ has $\mathsf{IP}$. Indeed, for every  $i\in W\subseteq \omega$ we have $\models \phi([a_i];[b_W])$ so $(a_i;b_W)\models \Sigma_{\phi}$, so in particular $\models\phi'(a_i;b_W)$. Similarly $\models\psi'(a_i;b_W)$ for $i\notin W$.
 \end{proof}

\begin{pr}\label{pr:APNIP}
Let $T$ be a Robinson theory. If every quantifier-free positive formula in $T$ is $\mathsf{NIP}$,  then $T$ is $\mathsf{NIP}$. In fact, the assumption that quantifier-free formulas are $\mathsf{NIP}$ may be weakened to: for every quantifier-free $\phi$ and $\psi$, we have $\alt(\phi, \psi)<\infty$.
\end{pr}
\begin{proof}
  Suppose
  $\phi(x;y)$ has $\mathsf{IP}$ witnessed by $\psi(x;y)$ and parameters $(a_j)_{j<\omega}$, $(b_W)_{W\subseteq \omega}$ in some pec model $M$ of $T$. By \Cref{fact:robinfqe} there is a collection of positive quantifier-free formulas $(\phi_i(x;y))_{i\in I}$ such that $\phi(x;y)$ is equivalent to the infinitary conjunction $\bigwedge_{i\in I}\phi_i(x;y)$ in every pec model of $T$. Hence $\{\phi_i(x;y):i\in I\}\cup \{\psi(x;y)\}$ is inconsistent with $T$ (otherwise it would have a realisation in every sufficiently saturated pec model of $T$, and such a realisation would satisfy $\phi(x;y)\wedge \psi(x;y)$), so there is some finite $I_0\subseteq I$ such that $\bigwedge_{i\in I_0}\phi_i(x;y)\wedge \psi(x;y)$ is inconsistent with $T$. Then $\psi(x;y)$ and the parameters $(a_j)_{j<\omega}, (b_W)_{W\subseteq \omega}$ witness that the positive formula  $\bigwedge_{i\in I_0}\phi_i(x;y)$ has $\mathsf{IP}$.

  The last sentence is proven similarly, replacing $\psi$ with an infinite conjunction $\bigwedge_{i\in I'} \psi_i$.
\end{proof}
\begin{co}\label{co:robatnip}
If $T$ is Robinson and every atomic formula in $T$ is $\mathsf{NIP}$, then $T$ is $\mathsf{NIP}$. 
\end{co}
\begin{proof}
  By \Cref{pr:APNIP,pr:lor,pr:land}.
\end{proof}
Note that, at least a priori, even if $T$ is Robinson and $\phi(x;y)$ is atomic, in order to check that $\phi(x;y)$ is $\mathsf{NIP}$ one has to check all positive quantifier-free $\psi(x;y)$, not just the atomic ones.

\begin{prob}
   Construct a positive $\mathsf{NIP}$ theory where there is a positive $\phi(x;y)$, some $\abs y$-tuple $b$, and an indiscernible sequence alternating infinitely many times between $\phi(x;b)$ and $\neg \phi(x;b)$.
\end{prob}

\subsection{Examples}\label{sec:DLOA}
In this subsection we provide some examples of positive $\mathsf{NIP}$ theories.
First, we take a brief look at the positive theory of linear orders with an automorphism, and, more generally, at the positive theory of a linear order with a  $G$-action by automorphisms for any fixed group $G$. We provide a direct proof that these are $\mathsf{NIP}$ in positive logic; this can be also easily deduced from~\cite[Th\'eor\`eme~4.4.19]{dealdama}. Then, we use the amalgamation results obtained in \Cref{sec:amalgam} to conclude that the positive theory of ordered abelian groups with an automorphism is $\mathsf{NIP}$, as well as $\mathbb Q\mathsf{-OVSA}$ and $\mathbb R\mathsf{-OVSA}$. Before any of that, for the record, let us point out something obvious.
\begin{rem}
  A classical first-order theory is $\mathsf{NIP}$ if and only if, when viewed as a positive theory via  Morleyisation, it is $\mathsf{NIP}$ in the sense of \Cref{def:nip}.
\end{rem}

Let $L=\set{<,\sigma,\sigma\inverse}$ (working with $\le$ instead would result in a unique pec model, a singleton). Let $T$ be the positive theory of a linear order $<$ with an automorphism $\sigma$. Arguing as as in \Cref{rem:rovsarob}, we see that $T$ is Pillay. Moreover, the JCP is easily checked to hold. More generally, and for the same reason, if $G$ is any group and $L_G\coloneqq\set{<}\cup\set{g\cdot -: g\in G}$, then the theory $T_G$ of linear orders with an action of $G$ by automorphisms is Pillay.\footnote{Note that, in some cases, the action of $G$ will be trivial, since every torsion point of $G$ is forced to act identically by the presence of a linear order.} We work in a pec monster model $\monster$.

How do pec models of $T$ look like?
It is easy to see that they are ordered densely and without endpoints, and that the set of non-fixed points of is dense: a model containing an interval $(a,b)$ of fixed points may always be continued to one where $\sigma(x)>x$ for some $x$ with $a<x<b$. It is similarly shown that an interval $[a,b]$ contains a fixed point if and only if $b>\sigma^k(a)$ for every $k\in \mathbb Z$ (see e.g.~\cite{MC}). From this, it follows easily that $T$ does not have a model companion; this is a special case of \cite[Theorem 3.1]{kikyo_2000}, and also of the more general fact below.
\begin{fact}[{\cite[Theorem 1]{kikyo_strict_2002})}]
If $T'$ is a classical first-order theory which  has a completion with $\mathsf{SOP}$, and $T$ is the theory of models of $T'$ together with an automorphism $\sigma$, then $T$ has no model companion.
\end{fact}

\begin{lemma}\label{lemma:qf}
Atomic formulas in $T_G$ are $\mathsf{NIP}$.
\end{lemma}
\begin{proof}
Up to equivalence modulo $T_G$, atomic formulas are of the form $\phi(x;y)\coloneqq g\cdot x\mathrel\square y$ where $\square\in\set{<,>,=}$ and $g\in G$. Clearly $\operatorname{alt}(\phi, \psi)\leq 1$ for every  $\psi$ contradictory with $\phi$. Thus $\phi$ is $\mathsf{NIP}$ by \Cref{lem:alt}. 
\end{proof}

 \begin{pr}
For every group $G$, the positive theory $T_G$ is $\mathsf{NIP}$. 
 \end{pr}
\begin{proof}
By \Cref{lemma:qf,co:robatnip} it is enough to check that $T_G$ is Robinson, i.e.~that the category of models of $(T_G)_\mathrm{sub}$ has the AP. Observe that $T_G$ is universally axiomatised\footnote{In the classical sense. One cannot say that $<$ is a strict linear order with h-universal axioms.}, so $T_G=(T_G)_\mathrm{sub}$, hence this is the class of linear orders (not necessarily dense) with a $G$-action by automorphisms. For $G=(\mathbb{Z},+)$ the conclusion is stated in \cite[Proposition~3.4]{Lascar}, and for arbitrary $G$ this is \cite[Theorem~3.1(d)]{lipparini}. Since the argument is short, we reproduce it below for the sake of self-containedness. 
 Namely, if $A,B,C\models T_G$, and $A$ is a substructure of both $B$ and $C$,  with $B\setminus A$  disjoint as a set from $C\setminus A$ (which we may assume), then we can put a total order $<_D$ on $D\coloneqq B\cup C$ extending the given orders $<_B$ on $B$ and $<_C$ on $C$, by placing points in the correct cuts of $A$ and stipulating, for example, then for each such cut $p$ we have $p(B)<p(C)$. More precisely, 
 \begin{enumerate}
    \item  if there is $a\in A$ with $b<_B a$ and $a<_C c$, we set $b<_Dc$   
    \item  if there is $a\in A$ with $b>_B a$ and $a>_C c$,  we set $b>_Dc$, and
    \item  if $b,c\notin A$ and there is no $a$ as in any of the previous points, we set $b<_Dc$.
\end{enumerate}
Then $<_D$ is clearly a linear order on $D$, and the action of $G$ on $B$ and on $C$ induces a $G$-action on $(D,<_D)$ by order automorphisms. With this structure on $D$, both $B$ and $C$ become substructures of $D$, hence $D$ solves the amalgamation problem.
\end{proof}
We now move to the setting of automorphisms of oags and of ordered rational or real vector spaces. Recall from \Cref{rem:rovsarob} that these theories are all Robinson.\footnote{Also, all of these are universally axiomatised, again in the classical sense.} The JCP is easily verified by taking lexicographical products.

By \Cref{thm:redtoR} every  automorphism of an ordered abelian group extends to one of an ordered $\mathbb R$-vector space extending it. Therefore, as far as existentially closed models are concerned (e.g.~IVP over them, $\mathsf{NIP}$), working with ordered $\mathbb R$-vector spaces is not a real restriction. Amalgamation for models of $T_\mathrm{sub}$ is more delicate, since naming multiplication by scalars changes $T_\mathrm{sub}$ or, in other words, the notion of substructure, but we already took care of this in \Cref{thm:apqtor}.

\begin{lemma}\label{lemma:qfovsa}
Atomic formulas in $R\mathsf{-OVSA}$ are $\mathsf{NIP}$. The same holds for quantifier-free $(L_{\mathrm{oag}}\cup \set{\sigma, \sigma\inverse})$-formulas in the theory of ordered abelian groups with an automorphism.
\end{lemma}
\begin{proof}
 Up to syntactic manipulations each atomic formula is of the form $f(x)\mathrel\square b$ where $\square\in\set{<,>,=}$ and $f\in R[\sigma]$ (in the case of ordered abelian groups with an automorphism put $R\coloneqq\mathbb Z$). To conclude, invoke \Cref{lem:alt} and observe that for every indiscernible sequence $(a_i)_{i<\omega}$, the sequence $(f(a_i))_{i<\omega}$ is also indiscernible, in particular monotone or constant.
\end{proof}

\begin{co}\label{co:apthennip}
The positive theory of ordered abelian groups with an automorphism is $\mathsf{NIP}$, and so are $\mathbb R\mathsf{-OVSA}$ and  $\mathbb Q\mathsf{-OVSA}$.
\end{co}
\begin{proof}
By \Cref{thm:apqtor,co:robatnip,lemma:qfovsa}.
\end{proof}

\subsection{Invariant types}\label{sec:invtp}
Once again, unless otherwise stated, we let $T$ be a positive theory with the JCP and  $\monster$ a monster model. The letter $A$ will denote a small subset of $\monster$. Types over $\monster$ will be referred to as \emph{global types}. Since we want to deal with global types, we will need to mention sets and tuples outside of $\monster$. These are assumed to come from a larger monster model, which will go unnamed.

\begin{defin}\*
  \begin{enumerate}
  \item We call a global maximal type $p(x)$ \emph{invariant over $A$}
    or \emph{$A$-invariant} iff whenever $a\equiv_A a'$ and
    $\phi(x;y)$ is a positive formula we have
    $p(x)\proves \phi(x;a)\iff p(x)\proves \phi(x;a')$.  
  \item   For a pec model $M$, a formula $\phi(x;y)$ and a global $M$-invariant maximal type $p(x)$, we put
    \begin{multline*}
d_p\phi\coloneqq \{q\in S(M): (\exists c\models q)(\phi(x;c)\in p)\}\\=\{q\in S(M): (\forall  c\models q)(\phi(x;c)\in p)\}
\end{multline*}

  \item  We will denote the set of maximal
    positive types over $\monster$ which are invariant over $A$ by 
    $\invtypes(\monster, A)$.
  \item  For $p\in\invtypes(\monster, A)$ and  $B\supseteq \monster$, by $p\invext B$ we will  denote the set $\{\phi(x;a):a\in B, (\exists a'\equiv_A a)(\phi(x;a')\in p)\}$.
  \item If $p\in\invtypes(\monster, A)$ and $(I, <)$ is a linear order, an $I$-sequence $(a_i)_{i\in I}$ is a \emph{Morley sequence} in $p$ over $A$ iff for every $\bla i0<n\in I$ we have $a_{i_j}\models p\invext A a_{i_0}\ldots a_{i_{j-1}}$. The (uniquely determined) type of a Morley sequence over $A$ indexed on $I$ will be denoted by $\pow pI\invext A$.
\end{enumerate}
\end{defin}

\begin{rem}
If $p\in \invtypes(\monster, A)$, then every Morley sequence in $p$ over $A$ is $A$-indiscernible.
\end{rem}

\begin{pr}
Let $p\in\invtypes(\monster, A)$, and let $B\supseteq \monster$. Then $p\invext B$ is a maximal positive type over $B$.
\end{pr}
\begin{proof}
It is enough to show that for every finite $b\in B$ we have that  $p\invext A b$ is a maximal positive type. Let $\tilde b\in \monster$ be such that $b\equiv_A \tilde b$. Now $p\invext A\tilde b$ is a maximal positive type, and the two are conjugate.
\end{proof}

 In the classical first-order setting, by~\cite[Lemma~7.18]{simon_2015}, if $p$ is a global $A$-invariant type then  $(d_p\phi)(y)$ is a Borel subset of $S_y(A)$, regardless of the cardinality of the language. We do not know whether this holds in an arbitrary positive theory, but the proof of the aforementioned result may be adapted to prove the statements below. Recall that a \emph{constructible} subset of a topological space is a finite Boolean combination of closed (equivalently, open) sets. We consider $S_y(M)$ with the topology given by a subbasis of open sets consisting of the sets of the form $[\neg\phi(x)]$ for positive $\phi(x)$.
\begin{lemma}
Let $T$ be a positive $\mathsf{NIP}$ theory. Let $\phi(x;y)$ be a positive formula and $p\in S^\mathrm{inv}_x(\monster, M)$.  Then $d_p\phi$ is an intersection of at most $\abs T$ constructible  subsets of $S_y(M)$.
\end{lemma}
\begin{proof}
  Since $p$ is maximal, if $p(x)\proves \neg\phi(x;d)$ and $a\models p$ then $\models \psi(a;d)$ for some positive $\psi(x;y)$ which is contradictory with $\phi(x;y)$. By \Cref{lem:alt} there is a maximal $N=N_\psi$ such that there exist some $d$ and $(a_0,\ldots, a_{N-1})\models \pow pN\invext M$ alternating between $\phi(x;d)$ and $\psi(x;d)$ (i.e.~$\models \phi(a_i;d)\iff \models \psi(a_{i+1};d)$ for all $i<N-1$). Fix $d$ and $n\le N$ such that $n$ is maximal among the lengths of Morley sequences in $p$ over $M$ alternating between $\phi(x;d)$ and $\psi(x;d)$.  Suppose that $\psi(x;d)\in p$ but the alternation on some $(a_i)_{i<n}$ as above stops at $\phi$, i.e.~$\models \phi(a_{n-1};d)$. Then if we let $a_n\models p\invext M\bla a0,{n-1},d$ we violate maximality of $n$. So  $\models \psi(a_{n-1};d)$.

  Let $A_{n,\psi}(y)$ be the partial type saying that there is a Morley sequence in $p$ over $ M$ of length $n$ alternating between $\phi(x;y)$ and $\psi(x;y)$ and stopping at $\psi(x;y)$, and let $B_{n, \psi}(y)$ say the analogous thing but with ``stopping at $\phi$'' instead.  Then by the previous paragraph
  \[
    (d_p\phi)^\complement\subseteq\bigcup_{\psi\wedge \phi\implica \bot}\bigcup_{n\le N_\psi} [A_{n,\psi}(y)]\cap [B_{n+1,\psi}(y)]^\complement
  \]
  For the other inclusion suppose $\tp(d/M)\in [A_{n,\psi}(y)]\cap [B_{n+1,\psi}(y)]^\complement$. Let $(a_i)_{i<n}$ witness that $\tp(d/M)\in [A_{n,\psi}(y)]$. If we had $p(x)\proves \phi(x;d)$ then by taking $a_n\models p\invext M,\bla a0,{n-1},d$ we contradict the fact that $\tp(d/M)\in [B_{n+1,\psi}(y)]^\complement$.
\end{proof}
\begin{co}
  If $T$ is countable and $\mathsf{NIP}$ then invariant types are Borel-definable. That is, for every pec model $M$, positive formula $\phi(x;y)$ and $M$-invariant global positive type $p$, the set $d_p\phi$ is a Borel subset of $S_y(M)$.
  \end{co}
  \begin{question}
    Let $T$ be a positive $\mathsf{NIP}$ theory. Are invariant global types Borel-definable, regardless of $\abs T$?
  \end{question}

 The following characterisation of $\mathsf{NIP}$ in full first-order logic is standard, see e.g.~\cite[Proposition 2.43]{simon_2015} and \cite[Proposition 7.6]{casanovas}.

\begin{fact}\label{fact:counting}
  The following are equivalent for a complete first-order theory $T$.
  \begin{enumerate}
      \item \label{point:coh1} The theory $T$ is $\mathsf{NIP}$.
      \item \label{point:coh2}For every $M\models T$ and every $p\in S(M)$, there are at most $2^{\abs{M}+\abs{T}}$ global $M$-invariant extensions of $p$.
      \item \label{point:coh3}For every $M\models T$ and every $p\in S(M)$, there are at most $2^{\abs{M}+\abs{T}}$ global coheir extensions of $p$.
  \end{enumerate}
\end{fact}
We first verify that the implication $\ref{point:coh1}\allora\ref{point:coh2}$ still holds in positive logic.

\begin{pr}
  Assume $T$ is $\mathsf{NIP}$ and let $p,q$ be global maximal $A$-invariant types. If $\pow p\omega\invext A=\pow q\omega\invext A$ then $p=q$.
\end{pr}
\begin{proof}
  Suppose $p,q$ are counterexamples. By maximality there are mutually contradictory positive $\phi(x;d)\in p$ and $\psi(x; d)\in q$. By \Cref{rem:altpar} there is a Morley sequence $(a_i)_{i<n}$ in $p$ (equivalently, in $q$) over $A$ maximising $\operatorname{alt}(\phi,\psi,(a_i), d)$  among Morley sequences in $p$ over $A$. We may assume (by swapping the roles of $p$ and $q$ if necessary) that $a_{n-1}\models \phi(x;d)$. 
Let $a_n\models q\invext A (a_i)_{i<n}d$. By assumption $\models \psi(a_n;d)$. But by construction $(a_i)_{i\le n}$ is a Morley sequence in $p$ (and in $q$) over $A$, contradicting maximality of $\operatorname{alt}(\phi,\psi,(a_i), d)$.
\end{proof}

\begin{co}\label{co:counting_invariant}
If $T$ is $\mathsf{NIP}$ then, for every $A$, every $p\in S(A)$ has at most $2^{\abs{A}+\abs{T}}$ global $A$-invariant maximal extensions.
\end{co}
Thus the implication $\ref{point:coh1}\allora\ref{point:coh2}$ in \Cref{fact:counting}  still holds in positive logic. To examine the other implications,  we need to first specify a notion of coheir we want to work with. We will consider the following two natural notions of a coheir in positive logic, both specialising to the usual notion of a coheir in full first-order logic.

\begin{defin}
Let $M\subseteq \monster$ be a pec model and let $p(x)\in S(\monster)$. We will say that $p(x)$ is \begin{itemize}
    \item a \emph{positive coheir over $M$} (or a \emph{positive coheir extension} of $p\restr M$) iff every positive formula in $p(x)$ has a realisation in $M$;
    \item a \emph{negative coheir over $M$} (or a \emph{negative coheir extension} of $p\restr M$) iff for every positive formula $\phi(x)$ with parameters from $\monster$, if $\neg \phi(x)$ follows from $p(x)$, then $\neg\phi(x)$ has a realisation in $M$.
\end{itemize}
\end{defin}
\begin{rem}
If $M\models T$ is pec and $p$ is a global positive coheir over $M$, then $p$ is $M$-invariant.
\end{rem}
\begin{proof}
If not, then there are mutually contradictory positive formulas $\phi(x;y)$ and $\psi(x;y)$ and $a\equiv_M a'$ such that $\phi(x;a)\wedge\psi(x;a')\in p(x)$. As $p(x)$ is a coheir, there is some $m\in M$ with $\models \phi(m;a)\wedge \psi(m;a')$. This contradicts $a\equiv_M a'$.
\end{proof}
Thus,  $\ref{point:coh2}\allora\ref{point:coh3}$ in \Cref{fact:counting} holds in positive logic for positive coheirs.

While, at a first glance, the notion of a positive coheir may seem more natural to consider in positive logic, it has a drawback: if we only work with maximal types (see \Cref{rem:primetypes}), then a type $p\in S(M)$ may have no global positive coheirs. In fact, it may even have no global $M$-invariant extensions.
\begin{eg}\label{eg:no_inv}
The following theory was described in \cite[Section 4]{poizat_quelques_2010}. Consider the language $L=\{P_n,R_n:n<\omega\}\cup\{r\}$ where the $P_n$ and $R_n$ are unary relation symbols and $r$ is a binary relation symbol. Let $M = \{a_n,b_n:n<\omega \}$, with $a_0,b_0,a_1,b_1,\ldots$ pairwise distinct, be the $L$-structure in which $P_n$ is interpreted as $\{a_n,b_n\}$, the relation $R_n$ as the complement of $P_n$, and $r$ as the symmetric anti-reflexive relation $\{(a_n,b_n),(b_n,a_n):n<\omega\}$. Let $T$ be the  positive theory of the structure $M$. Then the models of $T$ are bounded (in fact, every pec extension of $M$ adds at most two new points), so in particular $T$ is thick (see \Cref{def:hausdorff-thick} below). However the unique non-algebraic maximal type over $M$ does not have any global positive coheir extensions nor $M$-invariant extensions. 
\end{eg}

On the other hand, for negative coheirs we have the following.

\begin{pr}\label{pr:negchex}
  Let $M\models T$ be pec and let $p\in S(M)$. Then $p$ has a global negative coheir extension.
\end{pr}
\begin{proof}
Let $\pi(x)\coloneqq p(x)\cup\{\phi(x;c):c\in \monster, \neg\phi(M;c)=\emptyset\}$. We claim that $\pi(x)$ is consistent: otherwise, there are $\psi(x)\in p(x)$ and $\phi_0(x;c_0),\ldots,\phi_{n-1}(x;c_{n-1})$ with $\neg\phi_i(M;c_i)=\emptyset$ for all $i<n$ such that $\psi(x)\wedge \bigwedge_{i<n}\phi_i(x;c_i)$ is inconsistent. Then for $m\in M$ realising $\psi(x)$ we must have that $\models \neg\phi_i(m;c_i)$ for some $i$, which contradicts the choice of $\phi_i(x;c_i)$.
Thus $\pi(x)$ is consistent and hence extends to a global maximal type $q(x)$. Clearly $q(x)$ is a negative coheir extension of $p(x)$.
\end{proof}

\begin{pr}\label{pr:IPcoheirs}
If $T$ has $\mathsf{IP}$, then for every $\lambda\geq \abs{T}$ there are a pec model $M$ of size $\lambda$ and $p\in S(M)$ which has at least $2^{2^{\lambda}}$ global negative coheir extensions.
\end{pr}
\begin{proof}
Let $\phi(x;y)$, $\psi(x;y)$, $(a_i)_{i<\lambda}$, and $(b_{W})_{W\subseteq \lambda}$ witness $\mathsf{IP}$. Let $M$ be a pec model of size $\lambda$ containing all $a_i$. For every ultrafilter $\mathcal {U}$ on $\mathscr P(\lambda)$, put $\pi_{\mathcal{U}}(x)\coloneqq \{\phi(x;b_W):W\in \mathcal{U}\}\cup\{\psi(x;b_W):W\notin \mathcal{U}\}\cup \{\chi(x;c):c\in \monster, \neg\chi(M;c)=\emptyset\}$. Every finite subset of $\pi_{\mathcal{U}}(x)$ is realised by $a_i$ for some $i$, so $\pi_{\mathcal{U}}(x)$ is consistent and extends to a global negative coheir $p$ over $M$, say $q_{\mathcal{U}}$. As there are $2^{2^{\lambda}}$ many ultrafilters on $\mathscr P(\lambda)$, by the pigeonhole principle some $2^{2^\lambda}$ of them have the same restriction $p$ to $M$. Clearly $p$ satisfies the conclusion.
\end{proof}
Thus,  $\ref{point:coh3}\allora\ref{point:coh1}$ in \Cref{fact:counting} still holds in positive logic for negative coheirs.

In the rest of this subsection, we will generalise \Cref{fact:counting} to the classes of semi-Hausdorff positive theories and of thick positive theories. These are mild assumptions which were used by Ben-Yaacov in his development of simplicity theory in positive logic~\cite{ben-yaacov_simplicity_2003, BYthick}, and which were used in~\cite{DK} for developing the theory of $\mathsf{NSOP}_1$ in positive logic.

\begin{defin}
\label{def:hausdorff-thick}
A positive theory $T$ is
\begin{itemize}
\item \emph{semi-Hausdorff} iff equality of types is type-definable, i.e.~there is a partial type $\Omega(x, y)$ such that $\tp(a) = \tp(b)$ if and only if $\models \Omega(a, b)$;
\item \emph{thick} iff being an indiscernible sequence is type-definable, i.e.~there is a partial type $\Theta((x_i)_{i < \omega})$ such that $(a_i)_{i < \omega}$ is indiscernible if and only if $\models \Theta((a_i)_{i < \omega})$.
\end{itemize}
\end{defin}
It follows easily from the definitions that every semi-Hausdorff theory is thick.
\begin{rem}
By~\cite[Theorem~8]{poizat_positive_2018},   every positive theory whose category of models with homomorphisms has the Amalgamation Property is \emph{Hausdorff}, a condition stronger than being semi-Hausdorff. In particular, all positive theories considered in \Cref{sec:DLOA} are semi-Hausdorff, hence thick.
\end{rem}
\begin{defin}
\label{def:lascar-distance}
Let $a, a'$ be two tuples and let $B$ be any parameter set. We write $d_B(a, a') \leq n$ iff there are $a = a_0, a_1, \ldots, a_n = a'$ such that for all $0 \leq i < n$ there is a $B$-indiscernible sequence starting with $a_i, a_{i+1}$. We write $a\equiv^\mathrm{Ls}_B a'$ if $d_B(a, a') \leq n$ for some $n<\omega$; in this case we say $a$ and $a'$ are \emph{Lascar-equivalent} over $B$.

A global maximal type is \emph{Lascar-invariant over $A$} iff whenever $a\equiv^\mathrm{Ls}_A a'$ and $\phi(x;y)$ is a positive formula we have
    $p(x)\proves \phi(x;a)\iff p(x)\proves \phi(x;a')$.  
\end{defin}
\begin{fact}[{\cite[Proposition 1.5]{BYthick}}]
\label{fact:thick-lascar-distance}
A positive theory is thick if and only if, for all $B$ and $n \in \omega$, the property ``$d_B(x, x') \leq n$'' is type-definable over $B$ (in variables $x,x'$).
\end{fact}

\begin{rem}\label{rem:invariance}
 Let  $M\models T$ be pec and let $p$ be a negative coheir over $M$.
 \begin{enumerate}
     \item\label{point:invt2}If $T$ is semi-Hausdorff, then $p$ is invariant over $M$.
     \item\label{point:invthick} If $T$ is thick, then $p$ is Lascar-invariant over $M$.
 \end{enumerate}
\end{rem}
\begin{proof}\*
  \begin{enumerate}
  \item 
Suppose $p(x)$ is not $M$-invariant, so for some $a\equiv_M a'$ there is some positive $\phi(x;y)$ with $\phi(x;a)\in p(x)$ and $\phi(x;a')\notin p(x)$. By semi-Hausdorfness, there is  a partial positive type $\pi(x,a,a')$ over $M$ such that for every $c$ we have $\tp(ca/M)=\tp(ca'/M)\iff c\models \pi(x,a,a')$. Then there is some $\psi(x)\in \pi(x,a,a')$ such that $\psi(x)\notin p(x)$.
So, as $p(x)$ is a negative coheir, there is some $m\in M$ with $\models \neg\psi(m)$. But that means $\tp(ma)\neq \tp(ma')$,  contradicting the choice of $a$ and $a'$.
\item Suppose $p(x)$ is not Lascar-invariant over $M$, so for some $a\equiv^\mathrm{Ls}_M a'$ with $n\coloneqq d_M(a,a')<\infty$ we have $p(x)\nvdash \pi(x,a,a')$, where  $\pi(x,a,a')$ is a partial positive type over $M$ such that for every $c$ we have $d_M(ca,ca')\leq n\iff c\models \pi(x,a,a')$ (such a type exists by thickness and \Cref{fact:thick-lascar-distance}). Then there is $\phi(x)\in \pi(x,a,a')$ with $\phi\notin p$. As $p$ is a negative coheir over $M$, there is some $m\in M$ with $\models \neg \phi(m)$. But, as $\phi(x)\in \pi(x,a,a')$, this means that $d_M(ma,ma')>n=d_M(a,a')$, a contradiction.\qedhere
\end{enumerate}
\end{proof}
Note that in   \Cref{rem:invariance}, point~\ref{point:invthick}, the conclusion cannot be strengthened to $M$-invariance of $p$ by \Cref{eg:no_inv} and \Cref{pr:negchex}.
\begin{co} If $T$ is semi-Hausdorff, the following are equivalent.
\begin{enumerate}
      \item\label{point:st21} The positive theory $T$ is $\mathsf{NIP}$.
      \item\label{point:st22} For every pec $M\models T$ and every $p\in S(M)$, there are at most  $2^{\abs{M}+\abs{T}}$ global $M$-invariant extensions of $p$.
      \item \label{point:st23} For every pec $M\models T$ and every $p\in S(M)$, there are at most $2^{\abs{M}+\abs{T}}$ global negative coheir extensions of $p$.
\end{enumerate}
\end{co}
\begin{proof}
The implication $\ref{point:st23}\allora\ref{point:st21}$
is immediate from \Cref{pr:IPcoheirs} and $\ref{point:st21}\allora\ref{point:st22}$  follows from \Cref{co:counting_invariant}. As for  $\ref{point:st22}\allora\ref{point:st23}$, it is a consequence of   \Cref{rem:invariance}, point~\ref{point:invt2}.
\end{proof}

In the more general case of a thick theory, we  add a sufficient saturation condition in point $\ref{point:st23}$. 
We write  $\lambda_T\coloneqq \beth_{(2^{\abs{T}})^+}$.

\begin{co} If $T$ is thick, the following are equivalent.
\begin{enumerate}
      \item \label{point:thick1} The positive theory $T$ is $\mathsf{NIP}$.
      \item \label{point:thick2}For every pec $M\models T$ and every $p\in S(M)$, there are at most $2^{\abs{M}+\abs{T}}$ global $M$-invariant extensions of $p$.
      \item \label{point:thick3}For every $\lambda_{T}$-saturated pec $M\models T$ and every $p\in S(M)$, there are at most $2^{\abs{M}+\abs{T}}$ global negative coheir extensions of $p$.
\end{enumerate}
\begin{proof}
As in the semi-Hausdorff case, $\ref{point:thick3}\allora\ref{point:thick1}$ and $\ref{point:thick1}\allora\ref{point:thick2}$  follow from \Cref{pr:IPcoheirs} and  \Cref{co:counting_invariant} respectively. As for 
$\ref{point:thick2}\allora\ref{point:thick3}$, by \cite[Lemma 2.20]{DK}, if $M$ is a $\lambda_{T}$-saturated pec model then a global positive type is Lascar-invariant over $M$ if and only if it is $M$-invariant. Hence the implication follows by \Cref{rem:invariance}, point~\ref{point:invthick}.
\end{proof}
\end{co}
\begin{rem}\label{rem:primetypes}
Whilst above we exclusively worked with maximal types, there is a more general notion of ``prime type'', considered in~\cite{haykazyan_spaces_2019}: a \emph{prime type} is a collection $p(x)$ of positive formulas such that, if $\phi(x)\vee \psi(x)\in p(x)$, then $\phi(x)\in p(x)$ or $\psi(x)\in p(x)$. 
Maximal types are precisely the collections of positive formulas satisfied by a tuple from a pec model, while prime types are the collections of positive formulas satisfied by a tuple from an \emph{arbitrary} model.
\end{rem}

\begin{rem}
The space of prime types over $\monster$ may be endowed with two dual topologies: one has the $[\phi(x)]$, for positive $\phi(x)$, as a subbasis of open sets, the other has the same sets as a subbasis of closed sets, see~\cite[Remark~3.7]{haykazyan_spaces_2019}. If we identify $M$ with the set of global maximal types realised in $M$, then a type over $\monster$ is a positive coheir over $M$ if and only if it lies in the closure of $M$ in the first topology, and is a negative coheir over $M$ if and only if it lies in the closure of $M$ in the second topology. 
\end{rem}
The reason why maximal negative coheirs always exist, while maximal positive ones may fail to, is that the set of maximal types is compact with respect to the second topology, but not necessarily with respect to the first one. Since both topologies on the whole space of prime types are compact, \emph{prime} positive coheirs always exist.
\begin{eg}
  Let $T$ be the positive theory of $M\coloneqq(\omega, \le, 0,1,2,\ldots)$, as in~\cite{poizat_quelques_2010}. Because there is a constant symbol for every element of $\omega$, we see that $M$ embeds in every model of $T$, and it is easy to see that no model of $T$ can have points strictly before $0$, or between two consecutive natural numbers. Because we are using $\le$ (as opposed to $<$)  every other model of $T$ may be continued to $\omega+1$ by sending every point larger than all natural numbers to $\omega$, hence $T$ has only two pec models, namely $M=\omega$ and the maximal pec model $\monster=\omega+1$. The maximal type at $+\infty$ over $M$ has only one prime positive coheir over $\monster$, the non-maximal type $\set{x\ge n: n\in \omega}$.
\end{eg}

\small

\newcommand{\etalchar}[1]{$^{#1}$}

\end{document}